\documentclass[12pt]{article}
\usepackage[]{amsmath,amssymb}
\usepackage{amscd}
\usepackage{latexsym}
\usepackage{cite}
\usepackage{amsthm}

\renewcommand{\labelenumi}{(\theenumi)}
\renewcommand{\theenumi}{\roman{enumi}}
\newtheorem{definition}{Definition}[section]
\newtheorem{theorem}[definition]{Theorem}
\newtheorem{lemma}[definition]{Lemma}
\newtheorem{corollary}[definition]{Corollary}

\newtheorem{example}[definition]{Example}

\newtheorem{problem}[definition]{Problem}
\newtheorem{note}[definition]{Note}
\newtheorem{assumption}[definition]{Assumption}
\newtheorem{proposition}[definition]{Proposition}

\begin{document}
\title{\bf  
Leonard triples of $q$-Racah type \\
}
\author{
Paul Terwilliger}
\date{}

\maketitle
\begin{abstract}
Let $\mathbb F$ denote a field, and let $V$
denote a vector space over $\mathbb F$ with finite positive
dimension. Pick a nonzero $q \in \mathbb F$ such that
$q^4 \not=1$, and let $A,B,C$ denote a Leonard triple on
$V$ that has $q$-Racah type. We show that there exist
invertible $W, W', W'' $ in
${\rm End}(V)$  such that
(i) $A$ commutes with $W$ and $W^{-1}BW-C$;
(ii) $B$ commutes with $W'$ and $(W')^{-1}CW'-A$;
(iii) $C$ commutes with $W''$ and $(W'')^{-1}AW''-B$.
Moreover each of $W,W', W''$ is unique up to multiplication by
a nonzero scalar in $\mathbb F$. We show that the three elements
$W'W, W''W', WW''$ mutually commute, and their product is
a scalar multiple of the identity. A number of related results are
obtained.

 \end{abstract}

\section{Introduction}

This paper is about a linear algebraic object called a Leonard triple
\cite{mlt}.
Before describing this object, we first describe
a more basic object called a Leonard pair
\cite{2lintrans}.
Let $\mathbb F$ denote a field, and let $V$ denote a vector space
over $\mathbb F$ with finite positive dimension.
Let ${\rm End}(V)$ denote the $\mathbb F$-algebra consisting of
the $\mathbb F$-linear maps from $V$ to $V$.
According to 
\cite[Definition~1.1]{2lintrans},
a Leonard pair on $V$ is an ordered pair of
maps in 
${\rm End}(V)$ 
such that for
each map, there exists a basis of $V$ with respect to which the
matrix representing that map is diagonal and the matrix representing
the other map is irreducible tridiagonal. 
As explained in 
\cite[Appendix~A]{2lintrans}, 
the Leonard pairs provide a linear algebra interpretation
of a theorem of Doug Leonard
\cite{leonard},
\cite[p.~260]{bannaiIto}
concerning the $q$-Racah polynomials and their
relatives 
in the Askey scheme.
The Leonard pairs are classified up to isomorphism
\cite[Theorem~1.9]{2lintrans} and described further in
\cite{nomzerodiag,
nom2006,
ter24,
ter2004,
ter2005,
ter2005b,
vidter}. 
For a survey 
see \cite{madrid}.
 In \cite{mlt} Brian Curtin introduced the concept of a Leonard triple
as  
 a natural generalization of
a Leonard pair. 
According to 
 \cite[Definition~1.2]{mlt},
a Leonard triple on $V$ is a 3-tuple of
maps in 
${\rm End}(V)$ 
such that for
each map, there exists a basis of $V$ with respect to which the
matrix representing that map is diagonal and the matrices representing
the other two maps are irreducible tridiagonal.
\medskip


\noindent We give some background on Leonard triples.
Their study began with Curtin's comprehensive
treatment of a special case, said to be modular.
By \cite[Definition~1.4]{mlt} a Leonard triple on $V$
is modular whenever 
for each element of the triple there exists an antiautomorphism
of ${\rm End}(V)$ that fixes that element and 
swaps the other two elements of the triple.
In
\cite[Section~1]{mlt}
the modular Leonard triples are classified 
up to isomorphism.
The paper \cite{curt2} gives a natural correspondance between
the modular Leonard triples, and a family of Leonard pairs said to have spin
\cite[Definition~1.2]{curt2}.
The spin Leonard pairs
describe the irreducible modules for the subconstituent algebra
of a 
 distance-regular graph whose
Bose-Mesner algebra contains a spin model
\cite{caughman,
curtNom}. 
We will say more about modular Leonard triples shortly.
The general Leonard triples have recently been
 classified up to isomorphism,
via the following approach.
Using the eigenvalues one breaks down the analysis
into four special cases, called
  $q$-Racah, Racah, Krawtchouk, and Bannai/Ito
\cite{gao1,
bohou2015}.
The Leonard triples are classified up to isomorphism
in 
\cite{huang1}
(for  $q$-Racah type);
\cite{gao}
(for Racah type);
\cite{kang}
(for Krawtchouk type);
\cite{bohou2015b} (for Bannai/Ito type and even diameter);
\cite{bohou2013} (for Bannai/Ito type and odd diameter).
Additional results on Leonard triples can be found in
\cite{bohou2013b,
huang2, 
koro,
nomzerodiag,
zitnik}
(for $q$-Racah type);
\cite{alnajjar,
liu,post}
(for Racah type);
\cite{balmaceda,
stefko}
(for Krawtchouk type);
\cite{brown,
banItodunkl,
genest,
genest2,
   bouhou2016,
bohou2014b,
tsu}
(for Bannai/Ito type).
We mention two attractive families of Leonard triples,
said to be totally bipartite or totally almost-bipartite (abipartite).
The totally bipartite Leonard triples of $q$-Racah type are described
in \cite{bohou2014}; these are closely related to the finite-dimensional
irreducible modules for the algebra $U_q(\mathfrak{so}_3)$
\cite{hav1, hav2, hav3}.
The totally bipartite/abipartite Leonard triples of Bannai/Ito type
are described in 
\cite{brown}.
\medskip

\noindent  Turning to the present paper,
we consider a Leonard triple $A,B,C$ on $V$ that has $q$-Racah type.
In order to motivate our results, assume for the moment that
$A,B,C$ is modular.
By \cite[Corollary~2.6]{curt2} there exist invertible
$W,W',W''$ in 
${\rm End}(V)$ such that
(i) $AW=WA$ and $BW=WC$;
(ii) $BW'=W'B$ and $CW'=W'A$;
(iii) $CW''=W''C$ and $AW''=W''B$.
Moreover by \cite[Lemma~3.6]{curt2} each of $W, W', W''$ is unique up to 
multiplication by a nonzero scalar in $\mathbb F$.
By \cite[Lemma~9.3]{mlt} there exists
an invertible $P \in {\rm End}(V)$ such that
each of $W'W$, $W''W'$, $WW''$ is a scalar multiple of $P$.
Moreover
\begin{eqnarray*}
P^{-1}AP=B, \qquad \quad  
P^{-1}BP=C, \qquad \quad
P^{-1}CP=A
\end{eqnarray*}
and
 $P^3$ is a scalar multiple of the identity.
We now summarize our results.
Dropping the modular assumption,
we show that there exist
invertible $W, W', W'' $ in
${\rm End}(V)$  such that
(i) $A$ commutes with $W$ and $W^{-1}BW-C$;
(ii) $B$ commutes with $W'$ and $(W')^{-1}CW'-A$;
(iii) $C$ commutes with $W''$ and $(W'')^{-1}AW''-B$.
Moreover each of $W,W', W''$ is unique up to multiplication by
a nonzero scalar in $\mathbb F$. We show that the three elements
$W'W, W''W', WW''$ mutually commute, and their product is
a scalar multiple of the identity. 
Define
$\overline A=W^{-1}BW-C$ and similary define
$\overline B$,
$\overline C$.
By construction 
$\overline A$
commutes with $A$; in fact $\overline A$ is a polynomial in $A$
and we describe this polynomial in several ways.
We describe what happens if one of $A$, $B$, $C$
is conjugated by one of
$W^{\pm 1}$, $(W')^{\pm 1}$, $(W'')^{\pm 1}$;
the result is a polynomial of degree at most $2$ in 
$A$, $B$, $C$,
$\overline A$,
$\overline B$,
$\overline C$.
We also describe what happens if one of
$A$, $B$, $C$
is conjugated by one of
$W^{\pm 2}$, $(W')^{\pm 2}$, $(W'')^{\pm 2}$;
the result is a polynomial of degree at most $3$ in $A$, $B$, $C$.
We indicate how conjugation by
$W^{\pm 2}$, $(W')^{\pm 2}$, $(W'')^{\pm 2}$ is related to
the Lusztig automorphisms recently discovered by 
Baseilhac and Kolb
\cite{baseilhac}.
Using some basic hypergeometric series identities,
we express each of 
$W^{\pm 1}$, $W^{\pm 2}$
as a polynomial in $A$.
We mentioned above that 
$W'W, W''W', WW''$ mutually commute.
To obtain this result we show that (i) each of
$W'W, W''W', WW''$
commutes with $A+B+C$; (ii)
the subalgebra of
${\rm End}(V)$ generated by $A+B+C$
contains every element of 
${\rm End}(V)$ that commutes with $A+B+C$.
Near the end of the paper we show that
$\overline A=
\overline B=
\overline C=0$ if and only if
$A,B,C$ is modular, 
and in this case we recover the results of Curtin
mentioned above.
\medskip

\noindent This paper is organized as follows. Sections 2, 3
contain preliminaries. In Section 4 we review some basic
facts about Leonard triples, and  in Section 5 we consider the 
Leonard triples of $q$-Racah type.
Section 6 contains some trace formulae that will be used later
in the paper.
In Sections 7, 8 we introduce the elements
$W$, 
$W'$, 
$W''$ and
$\overline A$,
$\overline B$,
$\overline C$.
Sections 9, 10 are about conjugation.
In Section 11 we express the elements $W^{\pm 1}$, 
$W^{\pm 2}$ as polynomials in $A$. In Sections 12, 13
we show that
$W'W$, 
$W''W'$, 
$WW''$ mutually commute, and their product is a scalar multiple of
the identity. Section 14 is about the case
$\overline A=
\overline B=
\overline C=0$.

\section{Preliminaries about the eigenvalues}

We now begin our formal argument.
Throughout this section the following notation and
assumptions
are in effect. Fix nonzero $a, q \in \mathbb F$
such that $q^4 \not=1$.  Fix an integer $d \geq 0$, and
define
\begin{eqnarray}
\label{eq:th}
\theta_i = a q^{2i-d}  + a^{-1} q^{d-2i} \qquad \qquad
(0 \leq i \leq d).
\end{eqnarray}

\begin{lemma}
\label{lem1}
For $0 \leq i,j\leq d$,
\begin{eqnarray*}
\theta_i - \theta_j = (1-q^{2j-2i})(a q^{2i-d}-a^{-1}q^{d-2j}).
\end{eqnarray*}
\end{lemma}
\begin{proof} Use
(\ref{eq:th}).
\end{proof}

\begin{lemma}
\label{lem2}
The scalars $\lbrace \theta_i \rbrace_{i=0}^d$ are mutually distinct,
if and only if the following hold:
\begin{enumerate}
\item[\rm (i)] $q^{2i}\not=1$ for $1 \leq i \leq d$;
\item[\rm (ii)] $a^2$ is not among $q^{2d-2}, q^{2d-4}, \ldots, q^{2-2d}$.
\end{enumerate}
\end{lemma}
\begin{proof} Use
Lemma
\ref{lem1}.
\end{proof}


\begin{lemma}
\label{lem:eachprof}
For $1 \leq i \leq d$,
\begin{eqnarray*}
\frac{q \theta_{i} - q^{-1} \theta_{i-1}}{q^2-q^{-2}} =   a q^{2i-d-1},
\qquad \qquad
\frac{q \theta_{i-1} - q^{-1} \theta_{i}}{q^2-q^{-2}} =   a^{-1} q^{d-2i+1}.
\end{eqnarray*}
\end{lemma}
\begin{proof} Use 
(\ref{eq:th}).
\end{proof}

\begin{corollary} 
\label{lem4}
For $0 \leq i,j\leq d$ such that $|i-j|=1$,
\begin{eqnarray*}
\frac{q \theta_i - q^{-1} \theta_j}{q^2-q^{-2}}
\,
\frac{q \theta_j - q^{-1} \theta_i}{q^2-q^{-2}}
= 1.
\end{eqnarray*}
\end{corollary}
\begin{proof} Use
Lemma \ref{lem:eachprof}.
\end{proof}

\begin{lemma} 
\label{lem:ti}
For $0 \leq i \leq d$ pick $0 \not=t_i \in \mathbb F$.
Then the following {\rm (i)--(iii)} are equivalent:
\begin{enumerate}
\item[\rm (i)] For $0 \leq i,j \leq d$ such that $|i-j|=1$,
\begin{eqnarray}
\frac{t_j}{t_i} + \frac{q \theta_i - q^{-1} \theta_j}{q^2-q^{-2}} = 0.
\label{eq:titj}
\end{eqnarray}
\item[\rm (ii)] For $1 \leq i \leq d$,
\begin{eqnarray*}
\frac{t_i}{t_{i-1}} = -a^{-1} q^{d-2i+1}.
\end{eqnarray*}
\item[\rm (iii)] There exists $0 \not=\varepsilon \in  \mathbb F$ such that
\begin{eqnarray*}
t_i = \varepsilon (-1)^i a^{-i} q^{i(d-i)} \qquad \qquad (0 \leq i \leq d).
\end{eqnarray*}
\end{enumerate}
\end{lemma}
\begin{proof}
${\rm (i)} \Leftrightarrow {\rm (ii)}$ Use
Lemma
\ref{lem:eachprof}.
\\
\noindent
${\rm (ii)} \Rightarrow {\rm (iii)}$  By induction on $i$.
\\
\noindent
${\rm (iii)} \Rightarrow {\rm (ii)}$  Routine.
\end{proof}

\noindent We now consider when $q+q^{-1}$ is included among
 $\lbrace \theta_i
\rbrace_{i=0}^d$.

\begin{lemma}
\label{lem:qqth}
For $0 \leq i \leq d$,
\begin{eqnarray*}
\theta_i - q-q^{-1} = (a-q^{d-2i+1})(a-q^{d-2i-1})q^{2i-d}a^{-1}.
\end{eqnarray*}
\end{lemma}
\begin{proof} Use
(\ref{eq:th}).
\end{proof}

\begin{lemma}
\label{lem:thetaqq} Assume
that $\lbrace \theta_i \rbrace_{i=0}^d$ are mutually
distinct. Then 
the following {\rm (i)--(iii)} hold.
\begin{enumerate}
\item[\rm (i)] Assume $a=q^{d+1}$. Then
$\theta_0 =q+q^{-1} $.
\item[\rm (ii)] Assume 
$a=q^{-d-1}$. Then
$\theta_d =q+q^{-1}$.
\item[\rm (iii)] Assume 
$a\not=q^{d+1}$ and
$a\not=q^{-d-1}$. Then
$\theta_i \not=q+q^{-1} $ for $0 \leq i \leq d$.
\end{enumerate}
\end{lemma}
\begin{proof} Use Lemmas
\ref{lem2},
\ref{lem:qqth}.
\end{proof}

\noindent
Replacing the sequence $\lbrace \theta_i \rbrace_{i=0}^d$
by its inversion
$\lbrace \theta_{d-i} \rbrace_{i=0}^d$ has the following effect. 
\begin{lemma} 
\label{lem:inv}
We have
\begin{eqnarray*}
\theta_{d-i} = a^{-1} q^{2i-d} + a q^{d-2i}
\qquad \qquad (0 \leq i \leq d).
\end{eqnarray*}
\end{lemma}
\begin{proof} Use
(\ref{eq:th}).
\end{proof}

\section{Preliminaries about linear algebra}

\noindent We will be discussing algebras. An algebra is meant to 
be associative and have a multiplicative identity 1.
A subalgebra has the same 1 as the parent
algebra. Pick an integer $d\geq 0$,
and let $V$ denote a vector space over $\mathbb F$
with dimension  $d+1$. Let
${\rm End}(V)$ denote the $\mathbb F$-algebra consisting of
the $\mathbb F$-linear maps from $V$ to $V$.
Let $I \in
{\rm End}(V)$ denote the identity map.
For $A \in {\rm End}(V)$ let $\langle A \rangle$
denote the $\mathbb F$-subalgebra of 
${\rm End}(V)$ generated by $A$.
The element $A$ is called 
{\it diagonalizable} whenever $V$ is spanned by the eigenspaces of $A$.
The element $A$ is called {\it multiplicity-free}
whenever $A$ is diagonalizable, and each eigenspace of $A$
has dimension 1. 
Assume $A$ is multiplicity-free,
and let $\lbrace V_i\rbrace_{i=0}^d$ denote an ordering
of the eigenspaces of $A$. For $0 \leq i \leq d$
let $\theta_i$ denote the eigenvalue of $A$ for $V_i$.
Note that $\lbrace \theta_i \rbrace_{i=0}^d$ are mutually
distinct and contained in $\mathbb F$.
For $0 \leq i \leq d$ define
$E_i \in {\rm End}(V)$ such that
$(E_i - I)V_i = 0 $ and
$E_i V_j=0$ if $j\not=i$ $(0 \leq j \leq d)$.
We call $E_i$ the {\it primitive idempotent} of $A$ for
$V_i$ (or 
$\theta_i$).
We have (i) $E_i E_j = \delta_{i,j} E_i$ $(0 \leq i,j\leq d)$;
(ii)  $I = \sum_{i=0}^d E_i$;
(iii)   $AE_i = \theta_i E_i = E_i A$ $(0 \leq i \leq d)$;
(iv) $A = \sum_{i=0}^d \theta_i E_i$;
(v) $V_i = E_iV$ $(0 \leq i \leq d)$;
(vi) ${\rm tr}(E_i) = 1$ $(0 \leq i \leq d)$,
where tr means trace.
Moreover
\begin{eqnarray}
\label{eq:ei}
  E_i=\prod_{\stackrel{0 \leq j \leq d}{j \neq i}}
            \frac{A-\theta_j I}{\theta_i-\theta_j}
	    \qquad \qquad (0 \leq i \leq d).
\end{eqnarray}
Consider the $\mathbb F$-subalgebra $\langle A \rangle $
of ${\rm End}(V)$ generated by $A$. Then
$\lbrace A^i \rbrace_{i=0}^d$
is a basis for the $\mathbb F$-vector space
$\langle A \rangle$,
and
$0 = \prod_{0 \leq i \leq d} (A-\theta_i I)$. 
Moreover 
$\lbrace E_i \rbrace_{i=0}^d$ is a basis for the $\mathbb F$-vector
space 
$\langle A \rangle$. Here is another basis for the 
$\mathbb F$-vector space 
$\langle A \rangle$:
\begin{eqnarray}
\label{eq:taubasis}
(A-\theta_0 I) 
(A-\theta_1 I) 
\cdots
(A-\theta_{i-1} I)
\qquad \qquad (0 \leq i \leq d).
\end{eqnarray}
In this basis the primitive idempotents look as follows.
By 
(\ref{eq:ei}) and 
\cite[Lemma~5.4]{nom2006} we have

\begin{eqnarray}
\label{eq:Eitau}
E_i = \sum_{j=i}^d 
\frac{
(A-\theta_0 I)
(A-\theta_1 I)
\cdots 
(A-\theta_{j-1} I)
}
{
(\theta_i - \theta_0) \cdots
(\theta_i - \theta_{i-1})
(\theta_i - \theta_{i+1}) \cdots
(\theta_i - \theta_{j})}
\end{eqnarray}
for $0 \leq i \leq d$.
Note that for $X \in {\rm End}(V)$ the following
are equivalent:
(i) $X \in 
\langle A \rangle$;
(ii) $AX=XA$; 
(iii) $E_i X E_j = 0 $ if $i \not=j$ $(0 \leq i,j\leq d)$;
(iv) $X = \sum_{i=0}^d E_i X E_i$.
For 
$X \in {\rm End}(V)$  
we have $E_iXE_i = {\rm tr}(XE_i)E_i$
 $(0 \leq i \leq d)$.
Let ${\rm Mat}_{d+1}(\mathbb F)$ denote the 
$\mathbb F$-algebra consisting
of
the $d+1$ by $d+1$ matrices that have all entries in $\mathbb F$.
We index the rows and columns by $0,1,\ldots, d$.
Let $\lbrace v_i \rbrace_{i=0}^d$ denote a basis for $V$.
For $A \in {\rm End}(V)$ and 
$X \in 
{\rm Mat}_{d+1}(\mathbb F)$, we say that $X$ {\it represents $A$ with
respect to 
 $\lbrace v_i \rbrace_{i=0}^d$}
whenever $Av_j = \sum_{i=0}^d X_{ij} v_i$ for $0 \leq j \leq d$.
A matrix  
$X \in 
{\rm Mat}_{d+1}(\mathbb F)$ is called
 {\it tridiagonal}
whenever each nonzero entry is on the diagonal, the subdiagonal,
or the superdiagonal. Assume $X$ is tridiagonal. Then $X$
is called {\it irreducible} whenever each entry on the
subdiagonal is nonzero, and each entry on the superdiagonal is
nonzero.

\section{Leonard triples}

\noindent In this section we recall the definition and basic facts
about Leonard triples.

\begin{definition} 
\label{def:lt}
\rm (See \cite[Definition~1.2]{mlt}.) 
Let $V$ denote a vector space over $\mathbb F$ with
finite positive dimension. By a {\it Leonard triple on $V$}
we mean a 3-tuple $A,B,C$ of elements in
${\rm End}(V)$ such that 
\begin{enumerate}
\item[\rm (i)] 
there exists a basis for $V$ with respect to which the
matrix representing $A$ is diagonal and the
matrices representing $B$ and $C$ are irreducible tridiagonal;
\item[\rm (ii)] 
there exists a basis for $V$ with respect to which the
matrix representing $B$ is diagonal and the
matrices representing $C$ and $A$ are irreducible tridiagonal;
\item[\rm (iii)] 
there exists a basis for $V$ with respect to which the
matrix representing $C$ is diagonal and the
matrices representing $A$ and $B$ are irreducible tridiagonal.
\end{enumerate}
We say that the Leonard triple $A,B,C$ is {\it over} $\mathbb F$.
By the {\it diameter} of $A,B,C$ we mean
$d:={\rm dim}(V)-1$.
\end{definition}

\begin{lemma} Let $A,B,C$ denote a Leonard triple on $V$.
Then each permutation of $A,B,C$ is a Leonard triple on $V$.
Moreover, any two of $A,B,C$ form a Leonard pair on $V$.
\end{lemma}

\begin{lemma} 
\label{lem:2gen}
{\rm 
(See \cite[Corollary~3.2]{ter24}.)}
 Let $A,B,C$ denote a Leonard triple on $V$. Then
any two of $A,B,C$ generate the
$\mathbb F$-algebra 
${\rm End}(V)$.
\end{lemma}

%

\begin{definition}\rm (See \cite[Definition~8.2]{mlt}.) 
 Let $A,B,C$ denote a Leonard triple on $V$.
Let $\cal V$ denote a vector space over $\mathbb F$ with finite
positive dimension, and let
  $\cal A, \cal B, \cal C$ denote a Leonard triple on $ \cal V$.
An {\it isomorphism of Leonard triples from
$A,B,C$ to $\cal A, \cal B, \cal C$} is an $\mathbb F$-linear bijection
$\sigma : V \to \cal V$ such that
${\cal A} \sigma = \sigma A$ and
${\cal B} \sigma = \sigma B$ and
${\cal C} \sigma = \sigma C$.
The Leonard triples 
$A,B,C$ and
$\cal A, \cal B, \cal C$ are called {\it isomorphic} whenever
there exists an isomorphism of Leonard triples from
$A,B,C$ to
$\cal A, \cal B, \cal C$.
\end{definition}

\noindent For the rest of this section
let $A,B,C$ denote a Leonard triple on $V$, as in Definition
\ref{def:lt}. By \cite[Lemma~1.3]{2lintrans} 
each of $A,B,C$ is multiplicity-free.
Let $\lbrace \theta_i \rbrace_{i=0}^d$ 
denote an ordering of the eigenvalues of $A$.
For $0 \leq i \leq d$ let $0 \not=v_i \in V$ denote an eigenvector of $A$ for 
$\theta_i$. 
Note that the sequence 
$\lbrace v_i \rbrace_{i=0}^d$  is a basis for $V$.
The ordering 
$\lbrace \theta_i \rbrace_{i=0}^d$  is called {\it standard}
whenever 
the basis $\lbrace v_i \rbrace_{i=0}^d$  
satisfies
Definition
\ref{def:lt}(i). Assume that
the ordering $\lbrace \theta_i \rbrace_{i=0}^d$  is standard.
Then by 
\cite[Lemma~2.4]{TD00}
the ordering 
 $\lbrace \theta_{d-i} \rbrace_{i=0}^d$  is also standard and
 no further ordering is standard. Similar comments apply to
 $B$ and $C$.
For the rest of this section let
$\lbrace \theta_i \rbrace_{i=0}^d$ 
(resp. $\lbrace \theta'_i \rbrace_{i=0}^d$) 
(resp. $\lbrace \theta''_i \rbrace_{i=0}^d$) 
denote a standard ordering of the eigenvalues for $A$
(resp. $B$) 
(resp. $C$). 
Let 
$\lbrace E_i \rbrace_{i=0}^d$ 
(resp. $\lbrace E'_i \rbrace_{i=0}^d$) 
(resp. $\lbrace E''_i \rbrace_{i=0}^d$) 
denote the corresponding orderings of their
primitive idempotents.

\begin{lemma}
\label{lem:tripleprod}
{\rm (See 
\cite[Lemma~3.3]{vidter}.)}
For $0 \leq i,j\leq d$ the products 
\begin{eqnarray*}
&&
E_i B E_j, \qquad \quad E'_i C E'_j, \qquad \quad E''_i A E''_j,
\\
&&
E_i C E_j, \qquad \quad E'_i A E'_j, \qquad \quad E''_i B E''_j
\end{eqnarray*}
are zero if 
 $|i-j|>1$ and nonzero if
 $|i-j|=1$.
\end{lemma}

\begin{lemma}
\label{lem:eigreq}
{\rm (See \cite[Theorem~1.9]{2lintrans}.)} 
The scalars
\begin{eqnarray*}
\frac{\theta_{i-2}-\theta_{i+1}}{\theta_{i-1}-\theta_i},
\qquad \quad
\frac{\theta'_{i-2}-\theta'_{i+1}}{\theta'_{i-1}-\theta'_i},
\qquad \quad 
\frac{\theta''_{i-2}-\theta''_{i+1}}{\theta''_{i-1}-\theta''_i}
\end{eqnarray*}
are equal and independent of $i$ for $2 \leq  i \leq d-1$.
\end{lemma}

\noindent In the next section we will describe a family of Leonard triples,
said to have $q$-Racah type. Roughly speaking, this family
corresponds to 
the ``most general'' solution to the constraints in
Lemma
\ref{lem:eigreq}.
\medskip

\noindent We now formally define the
 modular Leonard triples. We will use the
following concept. By an {\it antiautomorphism} of
${\rm End}(V)$  we mean an $\mathbb F$-linear bijection
$\dagger: 
{\rm End}(V)\to
{\rm End}(V)$  such that
$(XY)^\dagger = Y^\dagger X^\dagger$ for all 
$X,Y \in {\rm End}(V)$.

\begin{definition}
\label{def:mlt}
\rm (See \cite[Definition~1.4]{mlt}.)
The Leonard triple $A,B,C$ is called {\it modular}
whenever for each element among $A,B,C$ there exists an antiautomorphism
of ${\rm End}(V)$ that fixes that element and 
swaps the other two elements of the triple.
\end{definition}

\section{Leonard triples of $q$-Racah type}

\noindent
In this section we describe a family of Leonard triples, said
to have $q$-Racah type.
For the rest of this paper the following notation is in effect.
Fix an integer $d\geq 0$.
Fix a nonzero $q \in \mathbb F$ such that
$q^4 \not=1$.   Fix nonzero $a,b,c \in \mathbb F$.
For the rest of this paper the following assumption is in effect.
\begin{assumption}
\label{lem:nonz}
We assume:
\begin{enumerate}
\item[\rm (i)] $q^{2i}\not=1$ for $1 \leq i \leq d$;
\item[\rm (ii)] none of $a^2, b^2, c^2$ is among
$q^{2d-2}, q^{2d-4}, \ldots, q^{2-2d}$;
\item[\rm (iii)] none of $abc, a^{-1}bc, ab^{-1}c, abc^{-1}$
is among $q^{d-1}, q^{d-3}, \ldots, q^{1-d}$.
\end{enumerate}
\end{assumption}

\noindent In a moment we will describe a Leonard triple $A,B,C$ over
$\mathbb F$
such that
\begin{eqnarray}
&&A+\frac{q BC-q^{-1} CB}{q^2-q^{-2}} =\alpha_a I,
\label{eq:A}
\\
&&B+\frac{q CA-q^{-1} AC}{q^2-q^{-2}} =\alpha_b I,
\label{eq:B}
\\
&&C+\frac{q AB-q^{-1} BA}{q^2-q^{-2}} =\alpha_c I
\label{eq:C}
\end{eqnarray}
where
\begin{eqnarray}
&&
\alpha_a = \frac{(a+a^{-1})(q^{d+1}+q^{-d-1})+ (b+b^{-1})(c+c^{-1})}{q+q^{-1}},
\label{eq:AA}
\\
&&
\alpha_b = \frac{(b+b^{-1})(q^{d+1}+q^{-d-1})+ (c+c^{-1})(a+a^{-1})}{q+q^{-1}},
\label{eq:BB}
\\
&&
\alpha_c=\frac{(c+c^{-1})(q^{d+1}+q^{-d-1})+ (a+a^{-1})(b+b^{-1})}{q+q^{-1}}
\label{eq:CC}
\end{eqnarray}
and $d$ is the diameter.

\begin{definition}
\label{lem:eigvalform}
\rm
For $0 \leq i \leq d$ define
\begin{eqnarray}
\theta_i = a q^{2i-d}+a^{-1} q^{d-2i},
\qquad 
\theta'_i = bq^{2i-d}+b^{-1} q^{d-2i},
\qquad 
\theta''_i = cq^{2i-d}+c^{-1} q^{d-2i}.
\label{eq:eigvalform}
\end{eqnarray}
\end{definition}
\begin{lemma} We have
\begin{eqnarray*}
\theta_i \not=\theta_j,\qquad \quad
\theta'_i \not=\theta'_j,\qquad \quad
\theta''_i \not=\theta''_j \qquad \quad \mbox{if} \quad i \not=j,
\qquad \qquad (0 \leq i,j\leq d).
\end{eqnarray*}
\end{lemma}
\begin{proof} By Lemma
\ref{lem2} and
Assumption
\ref{lem:nonz}.
\end{proof}

\begin{definition}
\label{def:qRAC}
\rm
A Leonard triple $A,B,C$ over $\mathbb F$
is said to have {\it $q$-Racah type with Huang data $(a,b,c,d)$}
whenever the following {\rm (i), (ii)} hold:
\begin{enumerate}
\item[\rm (i)] the sequence 
 $\lbrace \theta_i \rbrace_{i=0}^d$ 
(resp. $\lbrace \theta'_i \rbrace_{i=0}^d$ )
(resp. $\lbrace \theta''_i \rbrace_{i=0}^d$ )
is a standard ordering of the eigenvalues for 
$A$ (resp. $B$) (resp. $C$);
\item[\rm (ii)] 
$A,B,C$ satisfy
{\rm (\ref{eq:A})--(\ref{eq:CC})}.
\end{enumerate}
\end{definition}

\begin{note}\rm The above definition of
$q$-Racah type is slightly different from
the one given in
\cite[Definition~12.1]{huang1}. We make the adjustment
in order to allow $d\leq 2$.
For $d\geq 3$ the two definitions are equivalent by
\cite[Theorem~16.4]{huang1}.
\end{note}

\begin{definition}
\rm
For $1 \leq i \leq d$ define
\begin{eqnarray*}
&&
\varphi_i = a^{-1} b^{-1} q^{d+1}
(q^i - q^{-i})(q^{i-d-1}-q^{d-i+1})
(q^{-i} - abc q^{i-d-1})(q^{-i} - abc^{-1} q^{i-d-1}),
\\
&&
\phi_i = a b^{-1} q^{d+1}
(q^i - q^{-i})(q^{i-d-1}-q^{d-i+1})
(q^{-i} - a^{-1}bc q^{i-d-1})(q^{-i} - a^{-1} bc^{-1} q^{i-d-1}).
\end{eqnarray*}
\end{definition}

\begin{lemma}
\label{lem:phinonzero}
The scalars $\varphi_i, \phi_i$ are nonzero for $1 \leq i \leq d$.
\end{lemma}
\begin{proof} By Assumption
\ref{lem:nonz}.
\end{proof}

\begin{proposition}
\label{lem:LTexist}
{\rm (See \cite[Section~16]{huang1}.)}
There exists a Leonard triple $A,B,C$ over $\mathbb F$
that has $q$-Racah type and Huang data
$(a,b,c,d)$.
Up to isomorphism this Leonard
triple 
is uniquely determined by $q$ and $(a,b,c,d)$.
In one basis $A,B$ are represented by
    \begin{eqnarray*}
         &&
	 \left(
	    \begin{array}{cccccc}
	      \theta_0 &  & & & & {\bf 0}  \\
	       1 & \theta_1 & &    & &   \\
		 & 1 & \theta_2  & &  &
		  \\
		  &&\cdot & \cdot &&
		    \\
		    & & & \cdot  & \cdot & \\
		     {\bf 0}  & &  & & 1 & \theta_d
		      \end{array}
		       \right),
\qquad \qquad
	 \left(
	    \begin{array}{cccccc}
	      \theta'_0 & \varphi_1  & & & & {\bf 0}  \\
	        & \theta'_1 & \varphi_2 &    & &   \\
		 &  & \theta'_2  & \cdot &  &
		  \\
		  && & \cdot & \cdot &
		    \\
		    & & &   &  \cdot  & \varphi_d \\
		     {\bf 0}  & &  & &  & \theta'_d
		      \end{array}
		       \right)
\end{eqnarray*}
respectively.
\noindent In another basis 
$A,B$ are  represented by
    \begin{eqnarray*}
         &&
	 \left(
	    \begin{array}{cccccc}
	      \theta_d &  & & & & {\bf 0}  \\
	       1 & \cdot & &    & &   \\
		 & \cdot & \cdot  & &  &
		  \\
		  && \cdot & \theta_2 &&
		    \\
		    & & &  1 & \theta_1 & \\
		     {\bf 0}  & &  & & 1 & \theta_0
		      \end{array}
		       \right),
\qquad \qquad
	 \left(
	    \begin{array}{cccccc}
	      \theta'_0 & \phi_1  & & & & {\bf 0}  \\
	        & \theta'_1 & \phi_2 &    & &   \\
		 &  & \theta'_2  & \cdot &  &
		  \\
		  && & \cdot & \cdot &
		    \\
		    & & &   &  \cdot  & \phi_d \\
		     {\bf 0}  & &  & &  & \theta'_d
		      \end{array}
		       \right)
\end{eqnarray*}
respectively. In either basis the matrix representing $C$ is 
found using  
{\rm (\ref{eq:C})}.
\end{proposition}
\begin{proof} First assume that $d\geq 3$ and $\mathbb F$ is
algebraically closed. Then the result follows from
 \cite[Theorem~16.4]{huang1}
and the discussion below \cite[Definition~16.2]{huang1}. 
One checks that the result remains valid if the two assumptions are
removed.
\end{proof}

\begin{note}
\label{note:HuangData}
\rm It can happen that two distinct Huang data sequences
correspond to
 isomorphic Leonard triples of $q$-Racah type. 
 In other words, a given Leonard triple of $q$-Racah type
can have more than one Huang data sequence. 
Assume the Leonard triple $A,B,C$ has $q$-Racah type and
Huang data $(a,b,c,d)$.
Then by \cite[Lemma~4.8]{huang1}, each of 
$(a^{\pm 1},b^{\pm 1},c^{\pm 1},d)$ is a Huang data for
$A,B,C$ and $A,B,C$ has no other Huang data.
\end{note}

\begin{definition}\rm Referring to the Leonard triple
$A,B,C$ in Definition
\ref{def:qRAC},
for $0 \leq i \leq d$ let 
$E_i$ (resp. $E'_i$)
(resp. $E''_i$)
denote the primitive idempotent of $A$ (resp. $B$)
(resp. $C$)
for $\theta_i$
(resp. $\theta'_i$)
(resp. $\theta''_i$).
\end{definition}

\noindent We mention some properties for Leonard triples of
$q$-Racah type.

\begin{lemma} 
\label{lem:sym} 
Assume the Leonard triple
$A,B,C$ has $q$-Racah type, with Huang data $(a,b,c,d)$.
Then the Leonard triple $B,C,A$ has 
$q$-Racah type, with Huang data $(b,c,a,d)$.
Moreover the Leonard triple $B,A,C$ has $q^{-1}$-Racah type,
with Huang data 
$(b,a,c,d)$.
\end{lemma}
\begin{proof} By (\ref{eq:A})--(\ref{eq:CC}) and Definition
\ref{def:qRAC}.
\end{proof}

\noindent 
Recall the notation 
\begin{eqnarray*}
\lbrack n \rbrack_q = \frac{q^n-q^{-n}}{q-q^{-1}} \qquad \qquad n \in \mathbb Z.
\end{eqnarray*}

\begin{lemma}
\label{lem:traceABC}
Assume the Leonard triple
$A,B,C$ has $q$-Racah type, with Huang data $(a,b,c,d)$.
Then the traces of $A,B,C$ are as follows.
\begin{eqnarray*}
{\rm tr} (A) = (a+a^{-1})\lbrack d+1 \rbrack_q,
\qquad
{\rm tr} (B) = (b+b^{-1})\lbrack d+1 \rbrack_q,
\qquad
{\rm tr} (C) = (c+c^{-1})\lbrack d+1 \rbrack_q.
\end{eqnarray*}
\end{lemma}
\begin{proof} 
Concerning $A$
we have ${\rm tr} (A) = \sum_{i=0}^d \theta_i$;
evaluate this sum using the formula for $\theta_i$ in
Definition
\ref{lem:eigvalform}.
\end{proof}

\begin{lemma}
\label{lem:AAB}
Assume that the Leonard triple 
$A,B,C$ has $q$-Racah type, with Huang data $(a,b,c,d)$.
Then
\begin{eqnarray*}
&&
A^2 B - (q^2+q^{-2})ABA+ BA^2 + (q^2-q^{-2})^2 B =  \alpha_b (q^2-q^{-2})^2 I 
- \alpha_c (q-q^{-1})(q^2-q^{-2})A,
\\
&&
B^2 C - (q^2+q^{-2})BCB+ CB^2 + (q^2-q^{-2})^2 C =  \alpha_c (q^2-q^{-2})^2 I 
- \alpha_a (q-q^{-1})(q^2-q^{-2})B,
\\
&&
C^2 A - (q^2+q^{-2})CAC+ AC^2 + (q^2-q^{-2})^2 A=  \alpha_a (q^2-q^{-2})^2 I 
- \alpha_b (q-q^{-1})(q^2-q^{-2})C
\end{eqnarray*}
and also
\begin{eqnarray*}
&&
A^2C  - (q^2+q^{-2})ACA+ CA^2 + (q^2-q^{-2})^2 C=  \alpha_c (q^2-q^{-2})^2 I 
- \alpha_b (q-q^{-1})(q^2-q^{-2})A,
\\
&&
B^2A  - (q^2+q^{-2})BAB+ AB^2 + (q^2-q^{-2})^2 A =  \alpha_a (q^2-q^{-2})^2 I 
- \alpha_c (q-q^{-1})(q^2-q^{-2})B,
\\
&&
C^2B  - (q^2+q^{-2})CBC+ BC^2 + (q^2-q^{-2})^2 B =  \alpha_b (q^2-q^{-2})^2 I 
- \alpha_a (q-q^{-1})(q^2-q^{-2})C.
\end{eqnarray*}
\end{lemma}
\begin{proof} To obtain the first equation,
eliminate $C$ in
(\ref{eq:B}) using 
(\ref{eq:C}).
To obtain the fourth equation,
eliminate $B$ in
(\ref{eq:C}) using 
(\ref{eq:B}).
The remaining equations are similarly obtained.
\end{proof}

\section{Some trace formulae}

\noindent Let $V$ denote a vector space over $\mathbb F$
with dimension $d+1$.
For the rest of this paper $A,B,C$ denotes a Leonard triple
on $V$, with $q$-Racah type and
Huang data $(a,b,c,d)$.


\begin{lemma}
\label{cor:threea}
For $0 \leq i \leq d$ we have
\begin{eqnarray*}
&&
\biggl( 
{\rm tr}(BE_i)\,
\frac{q+q^{-1} + \theta_i}{q+q^{-1}}-\alpha_c
\biggr)
\biggl( 
1-\frac{\theta_i }{q+q^{-1}}
\biggr)
= \alpha_b - \alpha_c,
\\
&&
\biggl( 
{\rm tr}(CE'_i)\,
\frac{q+q^{-1} + \theta'_i}{q+q^{-1}}-\alpha_a
\biggr)
\biggl( 
1-\frac{\theta'_i }{q+q^{-1}}
\biggr)
= \alpha_c - \alpha_a,
\\
&&
\biggl( 
{\rm tr}(AE''_i)\,
\frac{q+q^{-1} + \theta''_i}{q+q^{-1}}-\alpha_b
\biggr)
\biggl( 
1-\frac{\theta''_i }{q+q^{-1}}
\biggr)
= \alpha_a - \alpha_b
\end{eqnarray*}
and also
\begin{eqnarray*}
&&
\biggl( 
{\rm tr}(CE_i)\,
\frac{q+q^{-1} + \theta_i}{q+q^{-1}}-\alpha_b
\biggr)
\biggl( 
1-\frac{\theta_i }{q+q^{-1}}
\biggr)
= \alpha_c - \alpha_b,
\\
&&
\biggl( 
{\rm tr}(AE'_i)\,
\frac{q+q^{-1} + \theta'_i}{q+q^{-1}}-\alpha_c
\biggr)
\biggl( 
1-\frac{\theta'_i }{q+q^{-1}}
\biggr)
= \alpha_a - \alpha_c,
\\
&&
\biggl( 
{\rm tr}(BE''_i)\,
\frac{q+q^{-1} + \theta''_i}{q+q^{-1}}-\alpha_a
\biggr)
\biggl( 
1-\frac{\theta''_i }{q+q^{-1}}
\biggr)
= \alpha_b - \alpha_a.
\end{eqnarray*}
\end{lemma}
\begin{proof} In the first equation of
Lemma
\ref{lem:AAB}, multiply each term on the left and
right by $E_i$, and simplify using
$E_iA=\theta_i E_i = AE_i$.
In the resulting equation take the trace of each side,
 and simplify using
 ${\rm tr}(XY)= 
 {\rm tr}(YX)$ along with $E^2_i = E_i$ and ${\rm tr}(E_i)=1$. 
This yields 
\begin{eqnarray}
&&
{\rm tr}(BE_i)\,
\frac{(q+q^{-1}-\theta_i)(q+q^{-1}+\theta_i)}{q+q^{-1}}
= \alpha_b (q+q^{-1})-\alpha_c \theta_i.
\label{eq:prelim}
\end{eqnarray}
Rearranging the terms in
(\ref{eq:prelim}) we obtain the first equation
 of the present lemma.
The remaining equations of the present lemma are similarly
obtained.
\end{proof}

\begin{lemma}
\label{lem:alphadif}
We have
\begin{eqnarray}
&&\alpha_a - \alpha_b = 
 \frac{(b-a)(b-a^{-1})(c-q^{d+1})(c-q^{-d-1})b^{-1} c^{-1}}{q+q^{-1}},
\label{eq:diff1} 
 \\
&&\alpha_b - \alpha_c = 
 \frac{(c-b)(c-b^{-1})(a-q^{d+1})(a-q^{-d-1})c^{-1} a^{-1}}{q+q^{-1}},
\label{eq:diff2} 
 \\
&&\alpha_c - \alpha_a = 
 \frac{(a-c)(a-c^{-1})(b-q^{d+1})(b-q^{-d-1})a^{-1} 
 b^{-1}}{q+q^{-1}}.
\label{eq:diff3} 
 \end{eqnarray}
 \end{lemma}
\begin{proof}
Use  
(\ref{eq:AA})--(\ref{eq:CC}).
\end{proof}

\section{The element $W$}

We continue to discuss the Leonard triple $A,B,C$ on $V$
with $q$-Racah type and Huang data $(a,b,c,d)$.
Our next general goal is to find all the
invertible $W  \in {\rm End}(V)$
such that $A$ commutes with $W$ and $W^{-1}BW-C$.
We first record some results from linear algebra.
\begin{lemma} 
\label{lem:WA}
For $W \in {\rm End}(V)$ the following are equivalent:
\begin{enumerate}
\item[\rm (i)] $W$ commutes with $A$;
\item[\rm (ii)] there exist 
scalars $\lbrace t_i \rbrace_{i=0}^d$ in
$\mathbb F$ such that
\begin{eqnarray}
W = \sum_{i=0}^d   t_i E_i.
\label{eq:We} 
\end{eqnarray}
\end{enumerate}
\end{lemma}

\begin{lemma}
Assume $W \in {\rm End}(V)$  satisfies the equivalent
conditions {\rm (i), (ii)} in Lemma
\ref{lem:WA}. Then the following are equivalent:
\begin{enumerate}
\item[\rm (i)] $W$ is invertible;
\item[\rm (ii)] $t_i \not=0$ for $0 \leq i \leq d$.
\end{enumerate}
Suppose {\rm (i), (ii)}. Then
\begin{eqnarray}
W^{-1} = \sum_{i=0}^d t^{-1}_i E_i.
\label{eq:Wie}
\end{eqnarray}
\end{lemma}

\noindent  Assume
that $W \in {\rm End}(V)$ is invertible and commutes with $A$.
Consider the scalars $\lbrace t_i \rbrace_{i=0}^d$ from
Lemma
\ref{lem:WA}.
We now find necessary and sufficient
conditions on the $\lbrace t_i \rbrace_{i=0}^d$ for  $A$ to commute with
$W^{-1} B W - C$.

\begin{lemma}
\label{lem:EWE}
Assume that $W \in {\rm End}(V)$ is invertible and  commutes with $A$.
For $0 \leq i,j\leq d$ consider
\begin{eqnarray}
E_i (W^{-1}BW-C)E_j.
\label{eq:WBE}
\end{eqnarray}
\begin{enumerate}
\item[\rm (i)]
Assume $|i-j|>1$. Then  
{\rm (\ref{eq:WBE})} is equal to $0$.
\item[\rm (ii)]
Assume $|i-j|=1$. Then  
{\rm (\ref{eq:WBE})} is equal to $E_iBE_j$ times
\begin{eqnarray*}
\frac{t_j}{t_i} + \frac{q \theta_i - q^{-1} \theta_j}{q^2-q^{-2}},
\end{eqnarray*}
where $t_0, \ldots, t_d$ are from Lemma
\ref{lem:WA}.
\item[\rm (iii)] Assume $i=j$. Then
{\rm (\ref{eq:WBE})} is equal to $E_i$ times
\begin{eqnarray*}
{\rm tr}(BE_i) \,\frac{q+q^{-1}+\theta_i}{q+q^{-1}}-\alpha_c.
\end{eqnarray*}
\end{enumerate}
\end{lemma}
\begin{proof} 
 In the expression
(\ref{eq:WBE}), 
eliminate $C$ using
(\ref{eq:C}) and
eliminate $W, W^{-1}$ using
(\ref{eq:We}), 
(\ref{eq:Wie}). Simplify the result using
Lemma
\ref{lem:tripleprod}.
\end{proof}

\begin{proposition} 
\label{lem:WBCA}
Assume that  $W \in {\rm End}(V)$ is invertible and commutes with $A$.
Then the following are equivalent:
\begin{enumerate}
\item[\rm (i)] $A$ commutes with $W^{-1}BW-C$;
\item[\rm (ii)] the scalars
$\lbrace t_i \rbrace_{i=0}^d$ from Lemma
\ref{lem:WA}
satisfy the equivalent conditions
{\rm (i)--(iii)} in Lemma
\ref{lem:ti}.
\end{enumerate}
\end{proposition}
\begin{proof} 
${\rm (i)} \Rightarrow {\rm (ii)}$
We show that
the scalars
$\lbrace t_i \rbrace_{i=0}^d$ satisfy
Lemma \ref{lem:ti}(i).
For $0 \leq i,j\leq d$ such that $|i-j]=1$,
let $\alpha_{ij}$ denote the expression on the left in
(\ref{eq:titj}). We show $\alpha_{ij}=0$.
By assumption (i) and 
Lemma
\ref{lem:EWE}(ii) we have
\begin{eqnarray*}
0 = E_i(W^{-1}BW-C)E_j = E_iBE_j \alpha_{ij}.
\end{eqnarray*}
But 
 $E_iBE_j\not=0$ by
 Lemma
\ref{lem:tripleprod}, so $\alpha_{ij}=0$.
\\
\noindent
${\rm (ii)} \Rightarrow {\rm (i)}$
By
Lemma \ref{lem:ti}(i) and
Lemma \ref{lem:EWE}(i),(ii) we find that
the expression 
(\ref{eq:WBE}) is zero for all $0 \leq i,j\leq d$ such that $i\not=j$.
Therefore $A$ commutes with $W^{-1}BW-C$.
\end{proof}

\begin{proposition}
\label{prop:view2}
For $W \in {\rm End}(V)$ the following are equivalent:
\begin{enumerate}
\item[\rm (i)] $W$ is invertible, and $A$ commutes with
$W$ and $W^{-1}BW-C$;
\item[\rm (ii)] there exists $0 \not=\varepsilon \in \mathbb F$
such that
\begin{eqnarray*}
W = \varepsilon \sum_{i=0}^d (-1)^i a^{-i} q^{i(d-i)} E_i.
\end{eqnarray*}
\end{enumerate}
\end{proposition}
\begin{proof} By 
Lemma \ref{lem:ti}(iii) and
Proposition
\ref{lem:WBCA}.
\end{proof}

\begin{theorem}
There exists an invertible $W \in {\rm End}(V)$ such that
$A$ commutes with $W$ and $W^{-1}BW-C$. This $W$ is unique
up to multiplication by a nonzero scalar in $\mathbb F$.
\end{theorem}
\begin{proof} By 
Proposition \ref{prop:view2}.
\end{proof}

\section{The elements 
$\overline A$, 
$\overline B$, 
$\overline C$}

\noindent 
We continue to discuss the Leonard triple $A,B,C$ on $V$
with $q$-Racah type and Huang data $(a,b,c,d)$.
For the rest of this paper we adopt the
following notation.

\begin{definition}
\label{def:WWW}
\rm Define
\begin{eqnarray*}
&&W = \sum_{i=0}^d (-1)^i a^{-i} q^{i(d-i)} E_i,
\\
&&W' = \sum_{i=0}^d (-1)^i b^{-i} q^{i(d-i)} E'_i,
\\
&&W'' = \sum_{i=0}^d (-1)^i c^{-i} q^{i(d-i)} E''_i.
\end{eqnarray*}
\end{definition}
\noindent We emphasize a few points.
\begin{lemma}
\label{lem:Wi} Each of $W$, $W'$, $W''$ is invertible. Moreover
\begin{eqnarray*}
&&W^{-1} = \sum_{i=0}^d (-1)^i a^{i} q^{-i(d-i)} E_i,
\\
&&(W')^{-1} = \sum_{i=0}^d (-1)^i b^{i} q^{-i(d-i)} E'_i,
\\
&&(W'')^{-1} = \sum_{i=0}^d (-1)^i c^{i} q^{-i(d-i)} E''_i.
\end{eqnarray*}
\end{lemma}
\begin{proof} Use
Definition \ref{def:WWW}.
\end{proof}

\begin{lemma}
\label{lem:ABCcom}
The following {\rm (i)--(iii)} hold.
\begin{enumerate}
\item[\rm (i)]
$A$ commutes with $W$ and $W^{-1}BW-C$;
\item[\rm (ii)] $B$ commutes with $W'$ and $(W')^{-1}CW'-A$;
\item[\rm (iii)] $C$ commutes with $W''$ and $(W'')^{-1}AW''-B$.
\end{enumerate}
\end{lemma}
\begin{proof} Part (i) is by Proposition
\ref{prop:view2}. Parts (ii), (iii) are similarly obtained.
\end{proof}


\begin{definition}
\label{def:Aoverline}
\rm Define
\begin{eqnarray*}
{\overline A} = W^{-1}BW-C,
\qquad
{\overline B} = (W')^{-1}CW'-A,
\qquad
{\overline C} = (W'')^{-1}AW''-B.
\end{eqnarray*}
\end{definition}

\begin{lemma}
\label{lem:AbarAbrack}
We have
\begin{eqnarray*}
\overline A \in \langle A \rangle,
\qquad \qquad
\overline B \in \langle B \rangle,
\qquad \qquad
\overline C \in \langle C \rangle.
\end{eqnarray*}
\end{lemma}
\begin{proof} The element $A$ is multiplicity-free, and
$\overline A$ commutes
with $A$, so $\overline A \in \langle A \rangle$.
The remaining assertions are similarly obtained.
\end{proof}

\begin{lemma}
\label{lem:Aoverline2}
We have 
\begin{eqnarray*}
{\overline A} = B-WCW^{-1},
\qquad
{\overline B} = C-W'A(W')^{-1},
\qquad
{\overline C} = A-W''B(W'')^{-1}.
\end{eqnarray*}
\end{lemma}
\begin{proof} $\overline A$ commutes
with $A$, and $W \in \langle A \rangle$, so 
$\overline A$ commutes with $W$.
Therefore
$\overline A = W {\overline A} W^{-1}=    
W(W^{-1}BW-C)W^{-1}= B-WCW^{-1}$.
The remaining assertions are similarly obtained.
\end{proof}

\begin{lemma}
\label{lem:traceAbar}
We have
\begin{eqnarray*}
{\rm tr}({\overline A}) = {\rm tr}(B) - {\rm tr}(C),
\qquad
{\rm tr}({\overline B}) = {\rm tr}(C) - {\rm tr}(A),
\qquad
{\rm tr}({\overline C}) = {\rm tr}(A) - {\rm tr}(B).
\end{eqnarray*}
\end{lemma}
\begin{proof} Use Definition
\ref{def:Aoverline}.
\end{proof}

\begin{corollary}
\label{cor:traceAbar}
We have
\begin{eqnarray*}
&&{\rm tr}({\overline A}) = 
(b-c)(b-c^{-1})b^{-1} \lbrack d+1 \rbrack_q,
\\
&&{\rm tr}({\overline B}) = 
(c-a)(c-a^{-1})c^{-1} \lbrack d+1 \rbrack_q,
\\
&&{\rm tr}({\overline C}) = 
(a-b)(a-b^{-1})a^{-1} \lbrack d+1 \rbrack_q.
\end{eqnarray*}
\end{corollary}
\begin{proof}
Use Lemmas
\ref{lem:traceABC},
\ref{lem:traceAbar}.
\end{proof}

\begin{proposition} 
\label{prop:Abarinv}
We have
\begin{eqnarray*}
&&
{\overline A}  \biggl(I - \frac{A}{q+q^{-1}}\biggr)
= (\alpha_b - \alpha_c)I 
=
\biggl(I - \frac{A}{q+q^{-1}}\biggr) \overline A;
\\
&&
{\overline B}  \biggl(I - \frac{B}{q+q^{-1}}\biggr) 
= (\alpha_c - \alpha_a)I
=
 \biggl(I - \frac{B}{q+q^{-1}}\biggr)  \overline B;
\\
&&
{\overline C}  \biggl(I - \frac{C}{q+q^{-1}}\biggr) 
= (\alpha_a - \alpha_b)I
=
 \biggl(I - \frac{C}{q+q^{-1}}\biggr)  \overline C.
\end{eqnarray*}
\end{proposition}
\begin{proof} We verify the top equations in the proposition
statement. Since $A,\overline A$ commute it suffices to
show
\begin{eqnarray}
{\overline A}  \biggl(I - \frac{A}{q+q^{-1}}\biggr)
= (\alpha_b - \alpha_c)I.
\label{eq:tocheck}
\end{eqnarray}
Let $\Delta$ denote the left-hand side of
(\ref{eq:tocheck}) 
minus the right-hand side of
(\ref{eq:tocheck}).
We show $\Delta=0$.
We have
$\Delta \in \langle A \rangle$
by 
Lemma \ref{lem:AbarAbrack},
so
$\Delta = \sum_{i=0}^d E_i \Delta E_i$.
We show $E_i \Delta E_i = 0 $ for $0 \leq i \leq d$.
Let $i$ be given. We have
\begin{eqnarray}
E_i \Delta E_i &=& E_i \overline A E_i 
\biggl(1-\frac{\theta_i}{q+q^{-1}} \biggr)
-(\alpha_b-\alpha_c) E_i
\nonumber
\\
&=& 
E_i (W^{-1}BW-C)E_i \biggl(1-\frac{\theta_i}{q+q^{-1}}\biggr)
-(\alpha_b-\alpha_c) E_i.
\label{eq:rhs}
\end{eqnarray}
\noindent Evaluating 
(\ref{eq:rhs})
using 
Lemma
\ref{lem:EWE}(iii) and the first equation
in Lemma
\ref{cor:threea}, we routinely obtain
$E_i\Delta E_i=0$.
We have shown (\ref{eq:tocheck}), and we verified the
top equations in the proposition statement. 
The remaining equations
are similarly verified.
\end{proof}

\noindent Consider the top equations in
Proposition
\ref{prop:Abarinv}. 
As we seek to describe $\overline A$, it is tempting
to invert the element $I-A/(q+q^{-1})$. However this element
might not be invertible. We now investigate this
possibility.

\begin{lemma} 
\label{lem:AAinv}
The following are equivalent:
\begin{enumerate}
\item[\rm (i)] $I-\frac{A}{q+q^{-1}}$ is invertible;
\item[\rm (ii)] $a\not=q^{d+1}$ and $a\not=q^{-d-1}$.
\end{enumerate}
\end{lemma}
\begin{proof}
Use Lemma \ref{lem:thetaqq}.
\end{proof}

\noindent We now describe $\overline A$. Similar
descriptions apply to $\overline B$ and $\overline C$.

\begin{proposition}
\label{prop:oA}
The element $\overline A$ is described as follows.
\begin{enumerate}
\item[\rm (i)] Assume $a=q^{d+1}$. Then
\begin{eqnarray}
{\overline A} = (b-c)(b-c^{-1})b^{-1} \lbrack d+1 \rbrack_q E_0.
\label{eq:Aover}
\end{eqnarray}
\item[\rm (ii)] Assume $a=q^{-d-1}$. Then
\begin{eqnarray}
{\overline A} = (b-c)(b-c^{-1})b^{-1} \lbrack d+1 \rbrack_q E_d.
\label{eq:Aover2}
\end{eqnarray}
\item[\rm (iii)] Assume $a\not=q^{d+1}$ and
 $a\not=q^{-d-1}$. Then
\begin{eqnarray}
{\overline A} =  (\alpha_b -\alpha_c)\biggl(I-\frac{A}{q+q^{-1}}\biggr)^{-1}.
\label{eq:Aover3}
\end{eqnarray}
\end{enumerate}
\end{proposition}
\begin{proof} (i) 
Recall that
$\lbrace E_i \rbrace_{i=0}^d$ is a basis for
$\langle A \rangle$, and
$\overline A \in \langle A \rangle$ by Lemma
\ref{lem:AbarAbrack}.
So there exist scalars $\lbrace s_i \rbrace_{i=0}^d$ in $\mathbb F$
such that $\overline A = \sum_{i=0}^d s_i E_i$.
By Lemma
\ref{lem:alphadif} we have
$\alpha_b = \alpha_c$, so
the first equation in Proposition
\ref{prop:Abarinv} becomes
\begin{eqnarray}
\overline A \biggl(1- \frac{A}{q+q^{-1}}\biggr) = 0.
\label{eq:prodz}
\end{eqnarray}
From 
(\ref{eq:prodz}) 
we obtain
\begin{eqnarray*}
s_i \biggl(1-\frac{\theta_i}{q+q^{-1}} \biggr) = 0
\qquad \qquad (0 \leq i \leq d).
\end{eqnarray*}
In the above equation, 
for $1 \leq i \leq d$
the coefficient of $s_i$ is nonzero
by Lemma
\ref{lem:thetaqq}(i) and 
$\theta_i \not=\theta_0$.
Consequently $s_i=0$. By these comments $\overline A = s_0 E_0$.
In this equation, take the trace of each side and
use
${\rm tr}(E_0)=1$ to obtain
$s_0 = {\rm tr}(\overline A)$.
This and
Corollary
\ref{cor:traceAbar} imply
(\ref{eq:Aover}).
\\
\noindent (ii) Similar to the proof of (i) above.
\\
\noindent (iii) By 
the first equation in Proposition
\ref{prop:Abarinv}, along with
Lemma \ref{lem:AAinv}.
\end{proof}

\noindent Here is another view of $\overline A$.

\begin{proposition} The element $\overline A$ is described as follows.
\begin{enumerate}
\item[\rm (i)] Assume $a=q^{d+1}$. Then
$\overline A$ is equal to
\begin{eqnarray}
\label{eq:p0}
(b-c)(b-c^{-1})b^{-1} \lbrack d+1\rbrack_q 
\end{eqnarray}
times
\begin{eqnarray}
\label{eq:pp}
\sum_{i=0}^d
\frac{(A-\theta_0 I)(A-\theta_1 I) \cdots (A-\theta_{i-1} I)}
{
(\theta_0-\theta_1)
(\theta_0-\theta_2)
\cdots
(\theta_0-\theta_i)
}.
\end{eqnarray}
\item[\rm (ii)] Assume $a=q^{-d-1}$. Then
$\overline A$ is equal to
\begin{eqnarray}
\label{eq:p1}
(b-c)(b-c^{-1})b^{-1} \lbrack d+1\rbrack_q 
\end{eqnarray}
times
\begin{eqnarray}
\label{eq:p2}
\frac{(A-\theta_0 I)(A-\theta_1 I) \cdots (A-\theta_{d-1} I)}
{
(\theta_d -\theta_0)
(\theta_d -\theta_1)
\cdots
(\theta_d-\theta_{d-1})
}.
\end{eqnarray}
\item[\rm (iii)] Assume $a \not=q^{d+1}$ and
$a \not=q^{-d-1}$.
 Then
$\overline A$ is equal to
\begin{eqnarray}
\label{eq:p3}
\frac{(a-q^{-d-1})(b-c)(b-c^{-1})b^{-1}q^d}{a-q^{d-1}}
\end{eqnarray}
times 
\begin{eqnarray}
\label{eq:p4}
\sum_{i=0}^d 
\frac{(A-\theta_0 I)(A-\theta_1 I) \cdots (A-\theta_{i-1} I)}
{
(q+q^{-1}-\theta_1)
(q+q^{-1}-\theta_2)
\cdots (q+q^{-1}-\theta_i)}.
\end{eqnarray}
\end{enumerate}
\end{proposition}
\begin{proof} 
(i) In line
(\ref{eq:Aover}) evaluate $E_0$
using
(\ref{eq:Eitau}).
\\
\noindent (ii)  
In line
(\ref{eq:Aover2}) evaluate $E_d$
using
(\ref{eq:Eitau}).
\\
\noindent (iii) Recall the basis  
(\ref{eq:taubasis}) for $\langle A \rangle$. Since
$\overline A \in \langle A \rangle$, 
there exist
scalars $\lbrace \gamma_i \rbrace_{i=0}^d$ in 
$\mathbb F$ such that
\begin{eqnarray*}
\overline A = \sum_{i=0}^d \gamma_i 
(A-\theta_0I) (A-\theta_1I)\cdots (A-\theta_{i-1}I).
\end{eqnarray*}
In this equation multiply each side by $I-A/(q+q^{-1})$.
Evaluate the resulting equation using
Proposition 
\ref{prop:Abarinv}, to obtain
\begin{eqnarray}
\alpha_b - \alpha_c = \gamma_0 
 \biggl(1-\frac{\theta_0}{q+q^{-1}} \biggr) 
\label{eq:gamma0}
\end{eqnarray}
and
$\gamma_{i-1}= \gamma_{i}(q+q^{-1}-\theta_i)$ for $1 \leq i \leq d$.
Evaluating (\ref{eq:gamma0}) using
 Lemma
\ref{lem:qqth}
and (\ref{eq:diff2}), we find that $\gamma_0$
is equal to
(\ref{eq:p3}). The result follows.
\end{proof}

\section{Some results about conjugation}

\noindent We continue to discuss the Leonard triple
$A,B,C$ on $V$ with $q$-Racah type and Huang data $(a,b,c,d)$.
Recall the maps
$W$, $W'$, $W''$ from Definition
\ref{def:WWW}. In this section we work out what happens
if one of $A,B,C$ is conjugated by one of
$W^{\pm 1}$, $(W')^{\pm 1}$, $(W'')^{\pm 1}$.
We start with some observations about $W$; similar observations
apply to $W'$ and $W''$.

\begin{lemma}
\label{lem:WBWi}
We have
\begin{eqnarray}
&& 
WBW^{-1} - W^{-1} BW = \frac{AB-BA}{q-q^{-1}},
\label{eq:WBW}
\\
&& 
WCW^{-1} - W^{-1} CW = \frac{AC-CA}{q-q^{-1}}.
\label{eq:WCW}
\end{eqnarray}
\end{lemma}
\begin{proof} Concerning 
(\ref{eq:WBW}), we have
\begin{eqnarray}
W^{-1} B W &=& \overline A + C
\nonumber
\\
&=& \overline A + \alpha_c I - \frac{qAB-q^{-1}BA}{q^2-q^{-2}}
\label{eq:WBW1}
\end{eqnarray}
and also
\begin{eqnarray}
WBW^{-1} &=& W \biggl( \alpha_b I - \frac{qCA-q^{-1}AC}{q^2-q^{-2}}\biggr) W^{-1}
\nonumber
\\
&=& \alpha_b I - \frac{q WCW^{-1} A - q^{-1} A WCW^{-1}}{q^2-q^{-2}}
\nonumber
\\
&=& \alpha_b I - \frac{q (B-\overline A) A - q^{-1} A(B-\overline A)}{q^2-q^{-2}}
\nonumber
\\
&=& \alpha_b I - \frac{qBA-q^{-1} AB}{q^2-q^{-2}} + \frac{A \overline A}{q+q^{-1}}.
\label{eq:WBW2}
\end{eqnarray}
Now subtract 
(\ref{eq:WBW1}) from
(\ref{eq:WBW2}), and evaluate the result using the first equation
in Proposition
\ref{prop:Abarinv}.
This yields
(\ref{eq:WBW}). Concerning
(\ref{eq:WCW}), we have
\begin{eqnarray}
W C W^{-1} &=& B- \overline A
\nonumber
\\
&=& 
\alpha_b I - \frac{qCA-q^{-1}AC}{q^2-q^{-2}}
-\overline A 
\label{eq:WCW1}
\end{eqnarray}
and also
\begin{eqnarray}
W^{-1}CW &=& W^{-1} \biggl(\alpha_c I - \frac{qAB-q^{-1}BA}{q^2-q^{-2}}\biggr) W
\nonumber
\\
&=& \alpha_c I - \frac{q A W^{-1}BW - q^{-1}  W^{-1}BWA}{q^2-q^{-2}}
\nonumber
\\
&=& \alpha_c I - \frac{q A (\overline A+C)  - q^{-1} (\overline A+C)A }{q^2-q^{-2}}
\nonumber
\\
&=& \alpha_c I - \frac{qAC-q^{-1} CA}{q^2-q^{-2}} - \frac{A \overline A}{q+q^{-1}}.
\label{eq:WCW2}
\end{eqnarray}
Now subtract 
(\ref{eq:WCW2}) from
(\ref{eq:WCW1}), and evaluate the result using the first equation
in Proposition
\ref{prop:Abarinv}.
This yields
(\ref{eq:WCW}).
\end{proof}

\noindent Let $X$ denote one of
$A$,
$B$,
$C$.
In the next result we conjugate $X$ by
each of $W^{\pm 1}$,  $(W')^{\pm 1}$, $(W'')^{\pm 1}$.

\begin{proposition}
\label{prop:ABCtable}
We have
\bigskip

\centerline{
\begin{tabular}[t]{c|ccc}
   $X$ & $A$ & $B$ & $C$ 
\\
   $W^{-1} X W$ & $A$ & $C+\overline A$ & $B+\frac{CA-AC}{q-q^{-1}}-\overline A$
   \\
   $WXW^{-1}$ & $A$  &
   $C+\frac{AB-BA}{q-q^{-1}}+\overline A$
   &
   $B-\overline A$
\\
   $(W')^{-1} X W'$ &
    $C+\frac{AB-BA}{q-q^{-1}}-\overline B$ &
    $B$ & $A+\overline B$
 \\ 
  $W'X(W')^{-1}$ &
   $C-\overline B$ &
    $B$  & $A+\frac{BC-CB}{q-q^{-1}}+\overline B$
\\
   $(W'')^{-1} X W''$ &
    $B+\overline C$ &
    $A+\frac{BC-CB}{q-q^{-1}}-\overline C$
    & $C$  
 \\ 
  $W''X(W'')^{-1}$ &
    $B+\frac{CA-AC}{q-q^{-1}}+\overline C$
   &
   $A-\overline C$ &
    $C$  
\\
           \end{tabular}}

\end{proposition}
\begin{proof} Use Definition
\ref{def:Aoverline} and Lemmas
\ref{lem:AbarAbrack},
\ref{lem:Aoverline2},
\ref{lem:WBWi}.
\end{proof}

\section{More results about conjugation}

\noindent We continue to discuss the Leonard triple
$A,B,C$ on $V$ with  $q$-Racah type and Huang data $(a,b,c,d)$.
 Recall the maps
$W$, $W'$, $W''$ from Definition
\ref{def:WWW}. In this section we work out what happens
if one of $A,B,C$ is conjugated by one of
$W^{\pm 2}$, $(W')^{\pm 2}$, $(W'')^{\pm 2}$.

\begin{lemma}
\label{def:W2W2W2}
We have
\begin{eqnarray}
&&
W^2 = \sum_{i=0}^d a^{-2i} q^{2i(d-i)} E_i,
\qquad \qquad 
W^{-2} = \sum_{i=0}^d a^{2i} q^{-2i(d-i)} E_i,
\label{eq:W2form}
\\
&&
(W')^2 = \sum_{i=0}^d b^{-2i} q^{2i(d-i)} E'_i,
\qquad \qquad
(W')^{-2} = \sum_{i=0}^d b^{2i} q^{-2i(d-i)} E'_i,
\label{eq:Wp2form}
\\
&&(W'')^2 = \sum_{i=0}^d c^{-2i} q^{2i(d-i)} E''_i,
\qquad \qquad 
(W'')^{-2} = \sum_{i=0}^d c^{2i} q^{-2i(d-i)} E''_i.
\label{eq:Wpp2form}
\end{eqnarray}
\end{lemma}
\begin{proof} Use
Definition
\ref{def:WWW}.
\end{proof}

\noindent We mention a result about $W^2$; similar results apply to
$(W')^2$ and $(W'')^2$.
Recall the commutator $\lbrack X,Y\rbrack = XY-YX$.

\begin{lemma} 
\label{lem:W22}
We have
\begin{eqnarray}
&&
\label{eq:W2BW2}
W^2 B W^{-2} + W^{-2} B W^2 =
2B + \frac{\lbrack A, \lbrack A,B\rbrack \rbrack}{(q-q^{-1})^2},
\\
&&
W^2 C W^{-2} + W^{-2} C W^2 = 2C + 
\frac{\lbrack A, \lbrack A, C\rbrack \rbrack}{(q-q^{-1})^2}.
\label{eq:W2CW2}
\end{eqnarray}
\end{lemma}
\begin{proof} Concerning
(\ref{eq:W2BW2}), observe
\begin{eqnarray*}
&&W^2 B W^{-2} + W^{-2} B W^2 - 2B
\\
&& \qquad = \;\;
W (WBW^{-1} - W^{-1} B W)W^{-1} -
W^{-1} (WBW^{-1} - W^{-1} B W)W
\\
&& \qquad = \;\;
\frac{W\lbrack A,B\rbrack W^{-1}-W^{-1}\lbrack A,B\rbrack W}{q-q^{-1}}
\\
&& \qquad = \;\;
 \frac{\lbrack A, WBW^{-1} - W^{-1} BW \rbrack}{q-q^{-1}}
\\
&& \qquad = \;\;
 \frac{\lbrack A,\lbrack A,B\rbrack \rbrack}{(q-q^{-1})^2}.
\end{eqnarray*}
The proof of (\ref{eq:W2CW2}) is similar.
\end{proof}

\noindent Let $X$ denote one of
$A$,
$B$,
$C$.
In the next result we conjugate $X$ by
each of $W^{\pm 2 }$,  $(W')^{\pm 2}$, $(W'')^{\pm 2}$.

\begin{proposition}
\label{prop:ABCtableW2}
We have
\bigskip

\centerline{
\begin{tabular}[t]{c|ccc}
   $X$ & $A$ & $B$ & $C$ 
\\
   $W^{-2} X W^2$ & $A$ & $B+\frac{\lbrack C,A\rbrack}{q-q^{-1}}$ &
   $C-\frac{\lbrack A,B\rbrack}{q-q^{-1}} +
   \frac{\lbrack A, \lbrack A, C\rbrack\rbrack}{(q-q^{-1})^2} $
   \\
   $W^2XW^{-2}$ & $A$  &
   $B-\frac{\lbrack C,A\rbrack}{q-q^{-1}} +
   \frac{\lbrack A, \lbrack A, B\rbrack\rbrack}{(q-q^{-1})^2} $
  & 
   $C+\frac{\lbrack A,B\rbrack}{q-q^{-1}}$
\\
   $(W')^{-2} X (W')^2$ &
   $A-\frac{\lbrack B,C\rbrack}{q-q^{-1}} +
   \frac{\lbrack B, \lbrack B, A\rbrack\rbrack}{(q-q^{-1})^2} $
   & 
    $B$ & 
    $C+\frac{\lbrack A,B\rbrack}{q-q^{-1}}$ 
 \\ 
  $(W')^2X(W')^{-2}$ &
    $A+\frac{\lbrack B,C\rbrack}{q-q^{-1}}$ 
   & 
   $B$
   &
   $C-\frac{\lbrack A,B\rbrack}{q-q^{-1}} +
   \frac{\lbrack B, \lbrack B, C\rbrack\rbrack}{(q-q^{-1})^2} $
\\
   $(W'')^{-2} X (W'')^2$ &
    $A+\frac{\lbrack B,C\rbrack}{q-q^{-1}}$ 
   &
   $B-\frac{\lbrack C,A\rbrack}{q-q^{-1}} +
   \frac{\lbrack C, \lbrack C, B\rbrack\rbrack}{(q-q^{-1})^2} $
 & $C$ 
 \\ 
  $(W'')^2X(W'')^{-2}$ &
   $A-\frac{\lbrack B,C\rbrack}{q-q^{-1}} +
   \frac{\lbrack C, \lbrack C, A\rbrack\rbrack}{(q-q^{-1})^2} $
    &
    $B+\frac{\lbrack C,A\rbrack}{q-q^{-1}}$ 
    &
    $C$
\\
           \end{tabular}}

\end{proposition}
\begin{proof} Consider the rows of the above table that involve $W$.
To verify these rows,
repeatedly apply Proposition
\ref{prop:ABCtable} 
and use the fact that $W$ commutes with $\overline A$. Lemma
\ref{lem:W22}
can be used to simplify the calculation.
The rows of the table involving $W'$, $W''$ are similarly
verified.
\end{proof}

\noindent We mention an alternate way to express the
results in Proposition
\ref{prop:ABCtableW2}. We focus on the part involving $W$.

\begin{lemma}
\label{prop:Wm2BW2}
We have
\begin{eqnarray*}
&&
W^{-2} B W^2 = B + \frac{qA^2B-(q+q^{-1})ABA+q^{-1} BA^2}{(q-q^{-1})(q^2-q^{-2})},
\\
&&
W^2 B W^{-2} =
B + \frac{q^{-1}A^2B-(q+q^{-1})ABA+q BA^2}{(q-q^{-1})(q^2-q^{-2})}.
\end{eqnarray*}
\end{lemma}
\begin{proof} 
In the formula for $W^{-2}BW^{2}$ and
 $W^{2}BW^{-2}$ given in Proposition
\ref{prop:ABCtableW2},
eliminate $C$ using
 (\ref{eq:C}).
\end{proof}

\begin{lemma}
\label{prop:W2C}
We have
\begin{eqnarray*}
&&
W^{-2} C  W^2 = C+ \frac{qA^2 C-(q+q^{-1})ACA +q^{-1} CA^2}{(q-q^{-1})(q^2-q^{-2})}, 
\\
&&
W^{2} C  W^{-2} = C + \frac{q^{-1} A^2C-(q+q^{-1})ACA+q CA^2}{(q-q^{-1})(q^2-q^{-2})}.
\end{eqnarray*}
\end{lemma}
\begin{proof} 
In the formula for $W^{-2}CW^{2}$ and
 $W^{2}CW^{-2}$ given in Proposition
\ref{prop:ABCtableW2},
eliminate $B$ using
 (\ref{eq:B}).
\end{proof}

\begin{note}\rm  Lemmas
\ref{prop:Wm2BW2},
\ref{prop:W2C} indicate how $W^2$, $(W')^2$, $(W'')^2$  are related to
the Lusztig automorphisms considered by Baseilhac and Kolb
\cite{baseilhac}.
\end{note}

\section{The elements $W^{\pm 1}$, $W^{\pm 2}$ as a polynomial in $A$}

\noindent We continue to discuss the Leonard triple
$A,B,C$ on $V$ with $q$-Racah type and Huang data $(a,b,c,d)$.
 Recall the map
$W$ from Definition
\ref{def:WWW}.
In this section we express 
each of $W^{\pm 1}$, $W^{\pm 2}$ as a polynomial in $A$. 
First we recall some notation. For $x,t \in \mathbb F$,
\begin{eqnarray*}
(x;t)_n = (1-x)(1-xt)\cdots (1-xt^{n-1}) \qquad \qquad
n=0,1,2,\ldots
\end{eqnarray*}
We interpret $(x;t)_0=1$.

\begin{proposition} 
\label{lem:W1}
We have
\begin{eqnarray}
W = \sum_{i=0}^d \frac{(-1)^i q^{i^2}(A-\theta_0I)(A-\theta_1I)\cdots (A-\theta_{i-1}I)}
{(q^2;q^2)_i (aq^{1-d};q^2)_i},
\label{eq:Wsum1}
\end{eqnarray}
\begin{eqnarray}
W^{-1} = \sum_{i=0}^d \frac{(-1)^i a^i q^{i(i-d+1)}(A-\theta_0I)(A-\theta_1I)\cdots (A-\theta_{i-1}I)}
{(q^2;q^2)_i (aq^{1-d};q^2)_i}.
\label{eq:Wisum1}
\end{eqnarray}
\end{proposition}
\begin{proof} 
It is convenient to prove  
(\ref{eq:Wisum1}) first.
Each side of (\ref{eq:Wisum1}) is a polynomial in $A$.
For $0 \leq j \leq d$ we show that the two sides have the same
eigenvalue for $E_j$.  
This amounts to showing
\begin{eqnarray}
(-1)^j a^{j} q^{j(j-d)} =
\sum_{i=0}^d \frac{(-1)^i  a^i q^{i(i-d+1)} 
(\theta_j-\theta_0)(\theta_j-\theta_1)
\cdots (\theta_j-\theta_{i-1})}
{(q^2;q^2)_i (aq^{1-d};q^2)_i}.
\label{eq:sumri}
\end{eqnarray}
Evaluating 
(\ref{eq:sumri})
using 
Lemma \ref{lem1}, we find it
becomes the following special case of the $q^2$-Chu-Vandermonde identity
\cite[p.~236]{gr}:
\begin{eqnarray}
(-1)^j a^{j} q^{j(j-d)} =
 {}_2\phi_1 \Biggl(
 \genfrac{}{}{0pt}{}
  {q^{-2j}, \;a^2 q^{2j-2d}}
   {a q^{1-d}}
    \;\Bigg\vert \; q^2,\; q^2 \Biggr).
\label{eq:identity1}
\end{eqnarray}
We are using the basic hypergeometric series notation
\cite[p.~3]{gr}.
For 
(\ref{eq:Wsum1}) the proof is similar.
For $0 \leq j \leq d$ we show that the two sides have the same
eigenvalue for $E_j$.  
This amounts to showing
\begin{eqnarray}
(-1)^j a^{-j} q^{j(d-j)} =
\sum_{i=0}^d \frac{(-1)^i q^{i^2} (\theta_j-\theta_0)(\theta_j-\theta_1)
\cdots (\theta_j-\theta_{i-1})}
{(q^2;q^2)_i (aq^{1-d};q^2)_i}.
\label{eq:sumr}
\end{eqnarray}
Evaluating
(\ref{eq:sumr}) 
using 
Lemma \ref{lem1},
we find it becomes the following identity, which is just
(\ref{eq:identity1}) with $a,q$ replaced by $a^{-1},q^{-1}$.
\begin{eqnarray*}
(-1)^j a^{-j} q^{j(d-j)} =
 {}_2\phi_1 \Biggl(
 \genfrac{}{}{0pt}{}
  {q^{2j}, \;a^{-2} q^{2d-2j}}
   {a^{-1} q^{d-1}}
    \;\Bigg\vert \; q^{-2},\; q^{-2} \Biggr).
\end{eqnarray*}
\end{proof}

\begin{proposition} 
\label{lem:W2sum1}
We have
\begin{eqnarray}
W^{2} = \sum_{i=0}^d \frac{ a^{-i} q^{id}(A-\theta_0I)(A-\theta_1I)\cdots (A-\theta_{i-1}I)}
{(q^2;q^2)_i},
\label{eq:W2sum1}
\end{eqnarray}
\begin{eqnarray}
W^{-2} = \sum_{i=0}^d \frac{ (-1)^i a^{i} q^{i(i-d+1)}
(A-\theta_0I)(A-\theta_{1}I)\cdots (A-\theta_{i-1}I)}
{(q^2;q^2)_i}.
\label{eq:W2isum1}
\end{eqnarray}
\end{proposition}
\begin{proof} Similar to the proof of
Proposition \ref{lem:W1}.
First consider 
(\ref{eq:W2isum1}).
For $0 \leq j \leq d$ we show that the two sides of
(\ref{eq:W2isum1})
have the same
eigenvalue for $E_j$.  
This amounts to showing
\begin{eqnarray}
a^{2j} q^{2j(j-d)} =
\sum_{i=0}^d \frac{(-1)^i a^{i} q^{i(i-d+1)}
(\theta_j-\theta_0)(\theta_j-\theta_1)
\cdots (\theta_j-\theta_{i-1})}
{(q^2;q^2)_i}.
\label{eq:sum2iri}
\end{eqnarray}
Evaluating 
(\ref{eq:sum2iri}) 
using 
Lemma \ref{lem1}, we find it 
reduces to the following special case of the $q^2$-Chu-Vandermonde 
identity \cite[p.~236]{gr}:
\begin{eqnarray}
a^{2j} q^{2j(j-d)} =
 {}_2\phi_1 \Biggl(
 \genfrac{}{}{0pt}{}
  {q^{-2j}, \;a^{2} q^{2j-2d}}
   {0}
    \;\Bigg\vert \; q^{2},\; q^{2} \Biggr).
\label{eq:identity3}
\end{eqnarray}
Next consider 
(\ref{eq:W2sum1}).
For $0 \leq j \leq d$ we show that the two sides of
(\ref{eq:W2sum1})
have the same
eigenvalue for $E_j$.  
This amounts to showing
\begin{eqnarray}
a^{-2j} q^{2j(d-j)} =
\sum_{i=0}^d \frac{a^{-i} q^{id} (\theta_j-\theta_0)(\theta_j-\theta_1)
\cdots (\theta_j-\theta_{i-1})}
{(q^2;q^2)_i}.
\label{eq:sum2ri}
\end{eqnarray}
Evaluating
(\ref{eq:sum2ri}) 
using 
Lemma \ref{lem1}, we find it 
reduces to the following identity, which is just
(\ref{eq:identity3})
with $a,q$ replaced by 
$a^{-1}, q^{-1}$.
\begin{eqnarray*}
a^{-2j} q^{2j(d-j)} =
 {}_2\phi_1 \Biggl(
 \genfrac{}{}{0pt}{}
  {q^{2j}, \;a^{-2} q^{2d-2j}}
   {0}
    \;\Bigg\vert \; q^{-2},\; q^{-2} \Biggr).
\end{eqnarray*}
\end{proof}

\section{A product}

\noindent We continue to discuss the Leonard triple
$A,B,C$ on $V$ with  $q$-Racah type and Huang data $(a,b,c,d)$.
 Recall the maps
$W$, $W'$, $W''$ from Definition
\ref{def:WWW}.
In this section we compute the product
$(W'')^2(W')^2W^2$.

\begin{proposition} 
\label{prop:W2W2W2}
We have
\begin{eqnarray*}
(W'')^2 (W')^2 W^2 = (abc)^{-d}q^{d(d-1)} I.
\end{eqnarray*}
\end{proposition}
\begin{proof}
We first show that 
$(W'')^2 (W')^2 W^2$ commutes with $A$.
Since $W$ commutes with $A$, it suffices to show that
$(W'')^2 (W')^2$ commutes with $A$. From Proposition
\ref{prop:ABCtableW2} we obtain
$(W')^2 A (W')^{-2} =
(W'')^{-2} A (W'')^{2}$. Therefore
$(W'')^2 (W')^2$ commutes with $A$. 
Next we show that
$(W'')^2 (W')^2 W^2$ commutes with $C$. Since 
 $W''$ commutes with $C$, it suffices to show that
$(W')^2 W^2$ commutes with $C$. From Proposition
\ref{prop:ABCtableW2} we obtain
$W^2 C W^{-2} =
(W')^{-2} C (W')^{2}$. Therefore
$(W')^2 W^2$ commutes with $C$. 
We have shown that $(W'')^2 (W')^2 W^2$ commutes with $A$ and $C$. 
By Lemma
\ref{lem:2gen} the elements $A,C$ generate ${\rm End}(V)$.
Therefore 
 $(W'')^2 (W')^2 W^2$ is central in
${\rm End}(V)$.
Consequently there exists $\theta \in \mathbb F$ such that
$(W'')^2 (W')^2 W^2  = \theta I$.
We now compute $\theta$.
For the rest of this proof, we identify $A$ and $B$
with their matrix representations from the first display
Proposition \ref{lem:LTexist}.
 The matrix $C$ is obtained using 
(\ref{eq:C}).
 The matrices
$A,B,C$ are lower bidiagonal, upper bidiagonal, and tridiagonal,
respectively. By
(\ref{eq:W2form}) and the construction,
the matrix $W^2$ 
is lower triangular with $(i,i)$-entry $a^{-2i} q^{2i(d-i)}$
for $0 \leq i \leq d$. By
(\ref{eq:Wp2form}) and the construction,
the matrix $(W')^2$ 
is upper triangular with $(i,i)$-entry $b^{-2i} q^{2i(d-i)}$
for $0 \leq i \leq d$. 
Note that 
$(W')^2 W^2  = \theta (W'')^{-2}$. 
In this equation we compute the $(d,0)$-entry of each side.
For the left-hand side, the $(d,0)$-entry is equal to
the $(d,d)$-entry of $(W')^2$ times the $(d,0)$-entry of
$W^2$. The 
$(d,d)$-entry of $(W')^2$ is  $b^{-2d}$.
Next we compute the $(d,0)$-entry of $W^2$.
By
(\ref{eq:W2sum1})
the element $W^2$ is a polynomial in $A$ with degree $d$ and
$A^d$-coefficient 
$a^{-d}q^{d^2}/(q^2;q^2)_d$.
The matrix $A$ is lower bidiagonal with $(i,i-1)$-entry 1
for $1 \leq i \leq d$.
Therefore
the $(d,0)$-entry of $A^d$ is 1, and the $(d,0)$-entry
of $A^j$ is 0 for $0 \leq j \leq d-1$. By these comments
the $(d,0)$-entry of $W^2$ is equal to
$a^{-d}q^{d^2}/(q^2;q^2)_d$.
Next we compute the $(d,0)$-entry of $(W'')^{-2}$.
Applying  
(\ref{eq:W2isum1}) to $W''$, we find that
$(W'')^{-2}$ is a polynomial in $C$ with degree $d$ and
$C^d$-coefficient
$(-1)^d c^d q^d/(q^2;q^2)_d$.
We mentioned that $C$ is tridiagonal; for $1 \leq i \leq d$
we compute its $(i,i-1)$-entry.
Using (\ref{eq:C}) and the form of $A,B$ we find that
$C$ has $(i,i-1)$-entry 
$(q^{-1} \theta'_i-q\theta'_{i-1})(q^2-q^{-2})^{-1}$ which is equal 
to 
 $-b^{-1} q^{d-2i+1}$.
Denote this quantity by $\beta_i$.
The $(d,0)$-entry of $C^d$ is $\beta_1 \beta_2 \cdots \beta_d$
which is equal to $(-1)^d  b^{-d}$.
The $(d,0)$-entry of $C^j$ is $0$ for $0 \leq j \leq d-1$.
By these comments the $(d,0)$-entry of
$(W'')^{-2}$ is 
$b^{-d} 
 c^d q^d/(q^2;q^2)_d$.
Putting it all together we obtain
\begin{eqnarray*}
\frac{b^{-2d}
  a^{-d}q^{d^2}}{(q^2;q^2)_d}
  =
  \frac{ \theta 
b^{-d} 
 c^d q^d}{(q^2;q^2)_d},
\end{eqnarray*}
so $\theta= (abc)^{-d} q^{d(d-1)}$. The result follows.
\end{proof}

\section{The elements $W'W$, $W''W'$, $W W''$}

\noindent We continue to discuss the Leonard triple
$A,B,C$ on $V$ with $q$-Racah type and Huang data
$(a,b,c,d)$.
 Recall the maps
$W$, $W'$, $W''$ from Definition
\ref{def:WWW}.
In this section we consider the three elements
\begin{eqnarray}
\label{eq:three}
W'W, \qquad  W''W', \qquad W W''.
\end{eqnarray}
We show that these elements mutually commute, and their product
is $ (abc)^{-d}q^{d(d-1)}I$.

\begin{lemma} 
\label{lem:comABC}
Each of the elements
{\rm (\ref{eq:three})} commutes with
 $A+B+C$.
\end{lemma} 
\begin{proof} We show that $W'W$ commutes with $A+B+C$.
 Using Proposition
\ref{prop:ABCtable} we obtain
\begin{eqnarray*}
W(A+B+C)W^{-1} = A+B+C+\frac{AB-BA}{q-q^{-1}}=
(W')^{-1}(A+B+C)W'.
\end{eqnarray*}
It follows that $W'W$ commutes with $A+B+C$.
\end{proof}

\noindent Recall that $\langle A+B+C\rangle$ is the
$\mathbb F$-subalgebra of ${\rm End}(V)$ generated by $A+B+C$.
\begin{lemma} 
\label{lem:minpoly}
There exists $v \in V$ such that 
$\langle A+B+C\rangle v = V$.
\end{lemma}
\begin{proof} 
Let $\lbrace v_i\rbrace_{i=0}^d $ denote the 
first basis for $V$ referred to in Proposition
\ref{lem:LTexist}. We show that $v=v_0$ satisfies the
requirements of the present lemma.
With respect to the basis
 $\lbrace v_i\rbrace_{i=0}^d$
the matrices representing
$A$ and $B$ are lower bidiagonal and upper bidiagonal, respectively.
Moreover by (\ref{eq:C}) the matrix representing $C$
is tridiagonal. Let $X \in 
{\rm Mat}_{d+1}(\mathbb F)$ 
represent $A+B+C$ with respect to
$\lbrace v_i\rbrace_{i=0}^d$.
The matrix $X$ is tridiagonal; we show that $X$ is
irreducible.
Using 
 (\ref{eq:C}) 
we find that for $1 \leq i \leq d$
the $(i,i-1)$-entry of 
$X$ is $1-b^{-1}q^{d-2i+1}$, which is nonzero by
Assumption
\ref{lem:nonz}(ii). 
Similarly
the $(i-1,i)$-entry of $X$ is $\varphi_i(1-a^{-1}q^{d-2i+1})$, which is
nonzero by
Assumption \ref{lem:nonz}(ii) and Lemma
\ref{lem:phinonzero}.
By these comments $X$ is irreducible.
It follows that for $0 \leq i \leq d$ there exists a polynomial
$f_i$ with  coefficents in 
$\mathbb F$ and degree $i$ such that $f_i(A+B+C)v_0=v_i$.
The vectors 
$\lbrace v_i\rbrace_{i=0}^d$ span $V$ so
$\langle A+B+C \rangle v_0 = V$.
\end{proof}


\begin{corollary}  \label{cor:WWWW}
The subalgebra 
 $\langle A+B+C \rangle$ contains every element of
${\rm End}(V)$ that commutes with $A+B+C$.
In particular
 $\langle A+B+C \rangle$ contains each 
 of the elements
{\rm (\ref{eq:three})}.
\end{corollary}
\begin{proof} 
Assume $G \in {\rm End}(V)$ commutes with $A+B+C$.
We show that $G \in \langle A+B+C\rangle$.
Recall the vector $v$ from Lemma
\ref{lem:minpoly}, and
consider $Gv$.
By Lemma \ref{lem:minpoly} there exists 
$H \in \langle A+B+C\rangle$ such that $Gv=Hv$.
Note that $G-H$ commutes with $A+B+C$.
We may now argue
\begin{eqnarray*}
(G-H)V &=& (G-H)\langle A+B+C\rangle v
\\
&=& \langle A+B+C\rangle (G-H)v
\\
&=& \langle A+B+C\rangle 0
\\
&=& 0.
\end{eqnarray*}
Therefore $G=H \in \langle A+B+C\rangle$.
\end{proof}

\begin{theorem}
\label{prop:com}
The elements in
{\rm (\ref{eq:three})} mutually commute.
\end{theorem}
\begin{proof} By Corollary
  \ref{cor:WWWW} and since 
the algebra $\langle A+B+C \rangle$ is commutative.
\end{proof}

\begin{theorem} 
\label{prop:prod}
The product of the elements
{\rm (\ref{eq:three})} is equal to
$ (abc)^{-d}q^{d(d-1)}I$.
\end{theorem}
\begin{proof} Using
Proposition
\ref{prop:W2W2W2} and Theorem
\ref{prop:com}
we obtain
\begin{eqnarray*}
(W'W)
(W''W')
(WW'') &=& 
(WW'') (W''W') (W'W) \\
&=&
W (W'')^2 (W')^2 W^2 W^{-1} \\
&=& 
(abc)^{-d}q^{d(d-1)}W W^{-1}\\
&=& 
 (abc)^{-d}q^{d(d-1)}I.
\end{eqnarray*}
\end{proof}

\section{The case $\overline A=\overline B=\overline C=0$}

\noindent We continue to discuss the Leonard triple
$A,B,C$ on $V$ with $q$-Racah type and Huang data $(a,b,c,d)$.
 Recall the maps
 $\overline A$, $\overline B$, $\overline C$ from
 Definition
\ref{def:Aoverline}.
In this section we consider the case in which
 $\overline A=\overline B=\overline C=0$.

\begin{definition}\rm We define a binary relation
$\sim$ on $\mathbb F$ called {\it similarity}.
For $x,y \in  \mathbb F$,
$x\sim y$ whenever 
$x=y$ or $xy=1$. 
Note that $\sim$ is an equivalence relation.
\end{definition}

\begin{lemma} 
\label{lem:3part}
The following {\rm (i)--(iii)} hold:
\begin{enumerate}
\item[\rm (i)] $\overline A=0$ if and only if $b \sim c$;
\item[\rm (ii)] $\overline B=0$ if and only if $c \sim a$;
\item[\rm (iii)] $\overline C=0$ if and only if $a \sim b$.
\end{enumerate}
\end{lemma}
\begin{proof} (i) 
 First assume that 
$a\sim q^{d+1}$. We show $\lbrack d+1 \rbrack_q \not=0$.
Suppose 
$\lbrack d+1 \rbrack_q =0$.
Then $1=q^{2d+2}$, which implies
$d\geq 1$. Also $a=a^{-1}$, which implies $\theta_0=\theta_d$. 
This contradicts the fact that
$\lbrace \theta_i \rbrace_{i=0}^d $ are mutually distinct.
Therefore 
 $\lbrack d+1 \rbrack_q \not=0$.
 Now by (\ref{eq:Aover}),
(\ref{eq:Aover2}) we find
$\overline A=0$ iff 
  $b\sim c$.
Next assume that 
$a \not\sim q^{d+1}$. 
Using 
(\ref{eq:diff2}),
(\ref{eq:Aover3})
  we obtain
$\overline A=0$ iff 
  $b\sim c$.
\\
\noindent (ii), (iii) Similar to the proof of (i) above.
\end{proof}

\begin{corollary}
\label{cor:ABCZero}
The following are equivalent:
\begin{enumerate}
\item[\rm (i)]
 $\overline A=\overline B=\overline C=0$;
\item[\rm (ii)]
$a,b,c$ are mutually similar;
\item[\rm (iii)] the Leonard triple $A,B,C$ is modular in
the sense of Definition 
\ref{def:mlt}.
\end{enumerate}
\end{corollary}
\begin{proof}
${\rm (i)} \Leftrightarrow {\rm (ii)}$
By Lemma
\ref{lem:3part}.
\\
\noindent
${\rm (i)} \Rightarrow {\rm (iii)}$
Recall that $W A=AW$ and $W'B=BW'$.
Since $\overline A=0$ we have $W^{-1}BW=C$.
Since $\overline B=0$ we have $C=W'A(W')^{-1}$.
So $W^{-1}BW = W'A(W')^{-1}$. Now in the sense of
 \cite[Definition~1.2]{curt2}, the pair $A,B$
is a spin Leonard pair with associated Boltman pair
$W^{-1}, (W')^{-1}$. Now by
 \cite[Theorem~1.6]{curt2}, the Leonard triple $A,B,C$ is modular.
\\
${\rm (iii)} \Rightarrow {\rm (ii)}$
By \cite[Lemma~9.3]{mlt} the Leonard triples $A,B,C$ and $B,C,A$
are isomorphic.
The scalars $a,b,c$ are mutually similar by
Note
\ref{note:HuangData}
and Lemma
\ref{lem:sym}.
\end{proof}

\begin{assumption} 
\label{ass:abc}
\rm
For the rest of this section, assume
 $\overline A=\overline B=\overline C=0$. 
By 
Corollary \ref{cor:ABCZero} the scalars
$a,b,c$ are mutually similar.  Inverting the eigenvalue orderings
$\lbrace \theta_i \rbrace_{i=0}^d$,
$\lbrace \theta'_i \rbrace_{i=0}^d$,
$\lbrace \theta''_i \rbrace_{i=0}^d$
as necessary, we will assume that $a=b=c$.
\end{assumption}

\begin{lemma} \label{lem:eigsame}
Under Assumption
\ref{ass:abc} we have
$\theta_i = \theta'_i = \theta''_i$ for $0 \leq i \leq d$.
\end{lemma}
\begin{proof}  By
(\ref{eq:eigvalform}) and $a=b=c$.
\end{proof}

\noindent Recall the maps $W$, $W'$, $W''$ from
Definition
\ref{def:WWW}.

\begin{definition}
\label{def:P}
\rm Define $P=W'W$.
\end{definition}

\begin{lemma} 
\label{lem:ABCP}
Under Assumption
\ref{ass:abc} we have
\begin{eqnarray*}
P^{-1} A P = B, \qquad \quad
P^{-1} B P = C, \qquad \quad
P^{-1} C P = A.
\end{eqnarray*}
\end{lemma}
\begin{proof} Use Proposition
\ref{prop:ABCtable} and 
Definition
\ref{def:P}.
\end{proof}

\begin{lemma} \label{lem:EiP}
Under Assumption
\ref{ass:abc} the following holds for
$0 \leq i \leq d$:
\begin{eqnarray*}
P^{-1} E_i P = E'_i, \qquad \quad
P^{-1} E'_i P = E''_i, \qquad \quad
P^{-1} E''_i P = E_i.
\end{eqnarray*}
\end{lemma}
\begin{proof} Use 
(\ref{eq:ei}) along with
Lemmas 
\ref{lem:eigsame},
\ref{lem:ABCP}.
\end{proof}

\begin{lemma} \label{lem:WWWP}
Under Assumption
\ref{ass:abc} we have
\begin{eqnarray*}
P^{-1} W P = W', \qquad \quad
P^{-1} W' P = W'', \qquad \quad
P^{-1} W'' P = W.
\end{eqnarray*}
\end{lemma}
\begin{proof}
Use Definition
\ref{def:WWW} and
Lemma
\ref{lem:EiP} and $a=b=c$.
\end{proof}

\begin{proposition}
\label{prop:PPPE}
Under Assumption
\ref{ass:abc} the element $P$ is equal to each of
\begin{eqnarray*}
W'W, \qquad 
W''W', \qquad 
WW''.
\end{eqnarray*}
\end{proposition}
\begin{proof}
In the equations $P^{-1}W'P=W''$ and
$P^{-1}W''P=W$,
 eliminate $P$ using $P=W'W$.
 \end{proof}

\begin{corollary}
Under Assumption
\ref{ass:abc} we have
$P^3 = a^{-3d}q^{d(d-1)}I$.
\end{corollary}
\begin{proof} Use Theorem
\ref{prop:prod},
Proposition \ref{prop:PPPE}, and $a=b=c$.
\end{proof}

\section{Acknowledgments}
The author thanks Brian Curtin, Edward Hanson, and Kazumasa Nomura for giving this paper
a close reading and offering valuable suggestions.


\bigskip

\noindent Paul Terwilliger \hfil\break
\noindent Department of Mathematics \hfil\break
\noindent University of Wisconsin \hfil\break
\noindent 480 Lincoln Drive \hfil\break
\noindent Madison, WI 53706-1388 USA \hfil\break
\noindent email: {\tt terwilli@math.wisc.edu }\hfil\break

\end{document}